\pdfoutput=1 
\documentclass{amsart}

\usepackage{microtype} 

\usepackage[colorlinks, citecolor=violet, linkcolor=blue]{hyperref}

\usepackage{wasysym}
\usepackage{amsmath}
\usepackage{amssymb}
\usepackage{amsthm}
\usepackage{wrapfig}
\usepackage{epsfig}  
\usepackage[margin=1.3in]{geometry}
\usepackage{color}
\input{xy}
\xyoption{poly}
\xyoption{2cell}
\xyoption{all}

\usepackage{tikz,tikz-cd}
\usepackage{adjustbox}
\usetikzlibrary{matrix,arrows,backgrounds,shapes.misc,shapes.geometric,patterns,calc,positioning}
\definecolor{mygreen}{cmyk}{0.5,0,0.5,0.5}
\definecolor{mylightblue}{RGB}{20,90,220}
\definecolor{mypink}{RGB}{200,90,90}
\usetikzlibrary{decorations.markings}
\tikzset{->-/.style={decoration={
  markings,
  mark=at position #1 with {\arrow{stealth}}},postaction={decorate}}}

\newtheorem{thm}{Theorem}[section]
\newtheorem{prop}[thm]{Proposition}
\newtheorem{lemma}[thm]{Lemma}
\newtheorem{corollary}[thm]{Corollary}

\newtheorem{thmIntro}{Theorem}

\theoremstyle{definition}
\newtheorem{definition}[thm]{Definition}

\theoremstyle{remark}
\newtheorem{remark}[thm]{Remark}
\newtheorem{example}[thm]{Example}

\numberwithin{equation}{section}

\newcommand{\arxiv}[1]{\href{http://arxiv.org/abs/#1}{\texttt{arXiv:#1}}}
  
\newcommand{\calt}{\mathcal{T}}
\newcommand{\cald}{\mathcal{D}}
\newcommand{\calm}{\mathcal{M}}
\newcommand{\cale}{\mathcal{E}}

\newcommand{\ssi}{\Leftrightarrow}
\newcommand{\ot}{\leftarrow}

\newcommand{\za}{\alpha}
\newcommand{\zb}{\beta}
\newcommand{\zd}{\delta}
\newcommand{\ze}{\epsilon}
\newcommand{\zg}{\gamma}
\newcommand{\zG}{\Gamma}
\newcommand{\zl}{\lambda}
\newcommand{\zs}{\sigma}

\newcommand{\kb}{\Bbbk}

\DeclareMathOperator{\supp}{supp}
\newcommand{\Hom}{\textup{Hom}}
\newcommand{\add}{\textup{add}}
\newcommand{\rep}{\textup{rep}\,}

\newcommand{\Qbar}{\overline{Q}}
\newcommand{\Mbar}{\overline{M}}
\newcommand{\fbar}{\overline{f}}
\newcommand{\Pbar}{\overline{P}}
\newcommand{\Cbar}{\overline{C}}
\newcommand{\Ibar}{\overline{I}}
\newcommand{\caltbar}{\overline{\mathcal{T}}}
\newcommand{\cbar}{\overline{\mathcal{C}}}

\newcommand{\Ext}{\textup{Ext}}
\newcommand{\End}{\textup{End}}

\newcommand{\mar}{\textup{mar}(Q)}
\newcommand{\ga}{\Gamma_{P(Q)}}
\newcommand{\cc}{\mathcal{C}_{P(Q)}}
\newcommand{\ind}{\textup{ind}\,}

\begin{document}

\title{Cambrian combinatorics on quiver representations (type $\mathbb{A}$)}
\author{Emily Barnard}\address{Department of Mathematical Sciences, DePaul University, Chicago, IL 60614-3210, USA}\email{
e.barnard@depaul.edu}
\author{Emily Gunawan}
\address{Department of Mathematics, University of Oklahoma, 
Norman, OK 73019-3103, USA}
\email{egunawan@ou.edu}
\author{Emily Meehan}\address{School of Science, Technology, Accessibility, Mathematics, and Public Health, 
Gallaudet University, 
Washington, DC 20002-3600, USA}\email{emily.meehan@gallaudet.edu}
\author{Ralf Schiffler}\thanks{The fourth author was supported by the NSF grant  DMS-1800860 and by the University of Connecticut.}
\address{Department of Mathematics, University of Connecticut, 
Storrs, CT 06269-1009, USA}
\email{schiffler@math.uconn.edu}

\maketitle
\setcounter{tocdepth}{2}

\begin{abstract}
This paper presents a  geometric model of the Auslander--Reiten quiver of a type $\mathbb{A}$ quiver together with a stability function for which all indecomposable modules are stable.
We also introduce a new Catalan object which we call a maximal almost rigid representation. We show that its endomorphism algebra is a tilted algebra of type $\mathbb{A}$. 
We define a partial order on the set of maximal almost rigid representations and use our new geometric model to show that this partial order is a Cambrian lattice.
\end{abstract}

\tableofcontents
\section{Introduction}

Let $Q$ be a Dynkin quiver of type $\mathbb{A}_{n+2}$. Thus $Q$ is an oriented graph whose underlying unoriented graph is the Dynkin diagram of type $\mathbb{A}_{n+2}$. 
In \cite{Reading}, Reading introduced a family of lattices called Cambrian lattices which depend on a choice of Dynkin diagram together with an orientation of that diagram.
In this paper we focus on the type $\mathbb{A}$ case, and use the orientation given by $Q$.
In this context, Reading defined an $(n+3)$-gon $P(Q)$ and a surjective map 
\[\eta_Q\colon S_{n+1}\to \{\textup{triangulations of $P(Q)$}\}\]
from permutations to triangulations, 
following and modifying an explication~\cite[Sec.~4.3]{Reiner} of an iterated fiber-polytope construction given in~\cite{BS}. 
In the special case where $Q$ is linearly oriented, this map had been 
well-studied as a map to various Catalan objects --- see \cite{Tonks}, \cite[Sec.~9]{BW}, \cite[Sec.~1.5]{Stanley}. 

One of the objectives of this paper is to give a representation theoretic interpretation 
of the map~$\eta_Q$. 
In order to achieve this goal, we develop several other concepts and results along the way.

Inspired by the map $\eta_Q$, we define a category $\cc$ whose  objects are line segments in the polygon $P(Q)$ and whose morphisms are generated by pivots of line segments modulo mesh relations. We show that the category $\cc$ is equivalent to  the abelian category $\ind Q$ of indecomposable representations of the quiver $Q$.   
\begin{thmIntro}[Theorem \ref{thm 1}]\label{thm A}
 There is an equivalence of categories $F\colon\cc\to \ind Q$.
\end{thmIntro}

The reader familiar with the construction of the cluster category of type $\mathbb{A}_n$ as the category of diagonals in an $(n+3)$-gon in \cite{CCS} will notice a strong similarity between the two constructions. Note however that our construction here is different in several ways. For one, we include the boundary edges, which was not the case in \cite{CCS}, and secondly, our category is hereditary and contains no cycles.

On the other hand, the result of Theorem~\ref{thm A} is very similar to the construction of the derived category of $ Q$ by Opper, Plamondon and Schroll in \cite{OPS}. In that paper, the authors give a geometric model for the derived category of an arbitrary gentle algebra. Since the path algebra of a quiver $Q$ of type $\mathbb{A}$ is a  special kind of gentle algebra, the category of representations of $ Q$ can be recovered as a heart of the derived category in \cite{OPS}.  However, the two constructions have some fundamental differences. While our polygon $P(Q)$ is   homeomorphic to the surface of \cite{OPS}, our construction is rigid in the sense that it fixes the relative position of each vertex of the polygon in the plane. This rigidity is essential here, since we are using the geometric data of the slope of the line segments in $P(Q)$ to describe the morphisms in the category. 
  
The slope is also the key to our next result proving  that the functor $F$ naturally induces a stability condition on the category of representations of $Q$. First, a  central charge is obtained by considering  each oriented line segment in $P(Q)$ as a complex number. Then the stability function of an indecomposable module is given by the angle of the oriented line segment with the positive real axis. This stability function has the property that every indecomposable representation is stable.
Reineke conjectured in \cite{Reineke} that  every Dynkin  quiver admits  a stability function with that property and Apruzzese and Igusa proved the conjecture in type $\mathbb{A}$ in \cite{AI}.

\begin{thmIntro}[Theorem \ref{thm A1}] \label{thm Astab} 
 Let $Q$ be a Dynkin quiver of type $\mathbb{A}$. Then the function that associates to every indecomposable representation $M$ the angle of the corresponding oriented line segment $F^{{-1}}(M)$ is a stability function for which every indecomposable representation of $Q$ is stable. 
\end{thmIntro}

 Using the stability function, we extend our model to a model of the derived category of $\rep Q$ recovering the (special case of) the result in \cite{OPS}.

We then introduce a class of representations which we call \emph{maximal almost rigid} representations. 
We say that the direct sum of two indecomposable representations $M$ and $N$ is \emph{almost rigid} if they do not have any nonsplit extensions, or if all extensions between $M$ and $N$ are indecomposable. 
We say that a  representation $T$  is \emph{maximal almost rigid} if,  for every nonzero representation $M$, the representation $T\oplus M$ is not almost rigid. Let $\mar$ denote the set of all maximal almost rigid representations of $Q$. 

Our third main result is the following.

\begin{thmIntro}[Theorem \ref{thm 2}]\label{thm B} The functor $F$ induces a bijection 
 \[F\colon\{\textup{triangulations of $P(Q)$}\} \to \mar. 
 \]
\end{thmIntro}

By definition the notion of an almost rigid representation is a generalization of that of a \emph{rigid} representation. For quivers without oriented cycles, the \emph{maximal rigid representations} are precisely the \emph{tilting modules} over the path algebra of the quiver, and their endomorphism algebras are the \emph{tilted algebras} introduced in \cite{HR}.
On the other hand, the endomorphism algebras of triangulations in cluster categories from marked surfaces are cluster-tilted algebras, or more generally,  2-Calabi-Yau tilted algebras, \cite{CCS,BMR,S2,ABCP, L}.
It is shown in \cite{ABS} that if $C$ is a tilted algebra, then its trivial extension $C\ltimes\Ext^2_C(DC,C)$ is cluster-tilted, and every cluster-tilted algebra arises this way.

 Both tilted and cluster-tilted algebras are extensively studied and occupy a central place in the representation theory of finite dimensional algebras.
 It is therefore natural to consider the endomorphism algebras of maximal almost rigid representations. Our fourth main result is the following.
 
\begin{thmIntro}[Theorem \ref{thm 3}]\label{thm C} Let $Q$ be a Dynkin quiver of type $\mathbb{A}_{n+2}$ and $T$ a  maximal almost rigid representation of $Q$. 
Then the endomorphism algebra
 $C=\End_{\rep Q} T$ is a tilted algebra of type $\mathbb{A}_{2n+3}$. 
 Moreover, the quiver of its associated  cluster-tilted algebra 
  $C\ltimes\Ext^2_C(DC,C)$ is the adjacency quiver of the triangulation $F^{-1}(T)$. 
\end{thmIntro}

We then come back to our original goal to give a representation theoretic interpretation of Cambrian lattices and the map $\eta_Q$.
In general, a Cambrian lattice is a certain quotient of the weak order on a finite Coxeter group. The Tamari lattice is an example in type $\mathbb{A}$. The Cambrian lattice for type $\mathbb{A}$ is a partial order on triangulations 
of a polygon whose cover relations are given by diagonal flips. 
The Hasse diagram for each Cambrian lattice is isomorphic to the 1-skeleton of the generalized associahedron for the corresponding Coxeter group, see~\cite{RS}. 
Furthermore, the elements of each Cambrian lattice of a finite Coxeter group are in bijection with several families of objects related to cluster algebras.
For example, the clusters of the corresponding cluster algebra of finite type, the (isoclasses of) cluster-tilting objects in the corresponding cluster category,
the finitely generated torsion classes over the path algebra of the Dynkin quiver, as well as the non-crossing partitions associated with the quiver, see~\cite{IT}. 

We add a new item to this list by proving the following.
\begin{thmIntro}[Theorem \ref{thm 4}]\label{thm D}
If $Q$ is a Dynkin quiver of type $\mathbb{A}$, 
 the set of maximal almost rigid representations of $Q$ 
 with the covering relation given in Definition~\ref{def poset} is isomorphic to the Cambrian lattice coming from $Q$.
\end{thmIntro}

We also introduce the representation theoretic map $\eta_Q^{\rep}$ as the composition $F\circ \eta_Q$ of the map $\eta_Q$ with our functor $F$. We show that $\eta_Q^{\rep}$ is described via degenerations and extensions of representations of $Q$ of dimension vector $(1,1\ldots,1)$.

The paper is organized as follows. We start by recalling the construction of the polygon~$P(Q)$ and the map $\eta_Q$ in section~\ref{sect 2}.  In section \ref{sect 3}, we review several facts on representations of quivers of type~$\mathbb{A}$.
Section \ref{sect 4} is devoted to the construction of the  category $\cc$ of line segments in~$P(Q)$ and the proof of Theorem \ref{thm  A}. 
We introduce the stability function in section~\ref{sect stab} and prove Theorem~\ref{thm Astab}. In the same section, we extend our model to the derived category.
In section~\ref{sect 5}, we define maximal almost rigid representations and prove Theorem~\ref{thm B}, and in section~\ref{sect 6}, we study their endomorphism algebras and prove Theorem~\ref{thm C}. 
In the last section, we give our representation theoretic interpretations of the Cambrian lattice and the $\eta_Q$ map and prove Theorem~\ref{thm D}. We also produce the interpretation of the $\eta_Q$ map via degenerations and extensions in that section.
\section{Construction of the polygon $P(Q)$ and the map $\eta_Q$} \label{sect 2}

We recall the construction of the  $(n+3)$-gon $P(Q)$ from~\cite{Reading}, 
where $Q$ is a Dynkin quiver of type $\mathbb{A}_{n+2}$. 
Label the vertices $1,2,\ldots, n+2$ in linear order as in  Figure~\ref{fig:polygon6}(left). 
Let $[n+1]$ denote the set $\{1,2,\ldots,n+1\}$ and let $S_{n+1}$ be the symmetric group.

First, use $Q$ to partition the set $[n+1]$ into two sets, the \emph{upper-barred} integers $\overline{[n+1]}$ and the \emph{lower-barred} integers $\underline{[n+1]}$, as follows. 
Let $\overline{[n+1]}$ be the set of all vertices $i$ such that $i\to (i+1)$ is in $Q$
and
let $\underline{[n+1]}$ be the set of all vertices $i$ such that $i\ot (i+1)$ is in $Q$.

Next, associate to $Q$ an $(n+3)$-gon $P(Q)$ with vertex labels $0,1,2,\dots,n+2$. Draw the vertices $0,1,\dots, n+2$ in order from left to right so that: (1) the vertices $0$ and $n+2$ are placed on the same horizontal line $L$; (2)  the upper-barred vertices are placed above $L$; (3) the lower-barred vertices are placed below $L$. See Figure~\ref{fig:polygon6}.

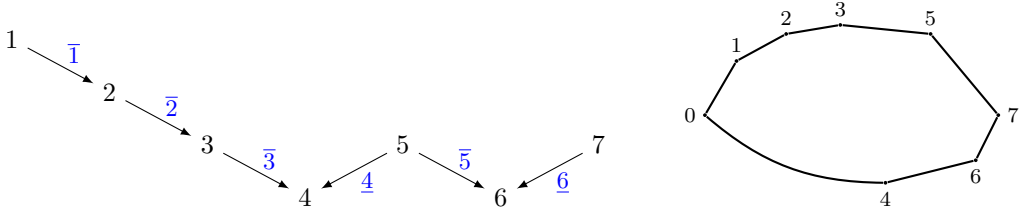
\begin{figure}
\begin{tikzpicture}[xscale=1.3,yscale=0.7,>=latex]
\def\posetedgecolor{blue}
\node(1) at (0,3) {{1}}; 
\node(2) at (1,2) {{2}}; 
\node(3) at (2,1) {{3}};  
\node(4) at (3,0) {{4}};  
\node(5) at (4,1) {{5}};   
\node(6) at (5,0) {{6}}; 
\node(7) at (6,1) {{7}}; 

\draw[->] (1) -- (2) node[\posetedgecolor,pos=0.7,above] {\small $\overline{1}$};

\draw[->]  (2) -- (3) node[\posetedgecolor,pos=0.7,above] {\small $\overline{2}$};

\draw[->] (3) -- (4) node[\posetedgecolor,pos=0.7,above] {\small $\overline{3}$}; 

\draw[<-] (4) -- (5) node[\posetedgecolor,pos=0.7,below] {\small $\underline{4}$}; 

\draw[->] (5) -- (6) node[\posetedgecolor,pos=0.7,above] {\small $\overline{5}$}; 

\draw[<-] (6) -- (7) node[\posetedgecolor,pos=0.7,below] {\small $\underline{6}$}; 
\end{tikzpicture}
\def\myxscale{0.6}
\def\myyscale{0.6}
\def\myedgesize{thick}
\begin{tikzpicture}[xscale=\myxscale,yscale=\myyscale]
\tikzstyle{every node}=[font=\footnotesize]
\node (0) at (0,0) [fill,circle,inner sep=0.5] {};
		\node (1) at (0.7,1.2) [opacity=1,fill,circle,inner sep=0.5] {};
		\node (2) at (1.8,1.8) [fill,circle,inner sep=0.5] {};
		\node (3) at (3,2) [fill,circle,inner sep=0.5]{}; 
		\node (4) at (4,-1.5) [fill,circle,inner sep=0.5] {};
		\node (5) at (5,1.8) [fill,circle,inner sep=0.5] {};
		\node (6) at (6,-1) [fill,circle,inner sep=0.5] {};
		\node (7) at (6.5,0) [fill,circle,inner sep=0.5] {};
		
		\draw (0) node[left] {$0$};
		\draw (1) node[above] {$1$};
		\draw (2) node[above] {$2$};
		\draw (3) node[above] {$3$};
		\draw (4) node[below] {$4$};
		\draw (5) node[above] {$5$};
		\draw (6) node[below] {$6$};
		\draw (7) node[right] {$7$};

\draw[\myedgesize] (0) to [bend right=20] (4);		
\draw[\myedgesize] (4) -- (6); 
\draw[\myedgesize] (6) -- (7);
\draw[\myedgesize] (0) -- (1) -- (2) -- (3) -- (5) -- (7);
\end{tikzpicture}

\caption{
Quiver $Q$ with $\underline{[6]} = \{ 4,6\}$
and  
$\overline{[6]} = \{ 1,2,3,5\}$	and 
the polygon $P(Q)$.}
\label{fig:polygon6}
\end{figure}
 
Given a permutation $\pi\in S_{n+1}$, we write $\pi$ in one-line notation as $\pi_1 \pi_2 \ldots \pi_{n+1}$, where $\pi_i=\pi(i)$ for $i\in [n+1]$. 
For each $i\in\{0,\dots,n+1\}$, define $\lambda_i(\pi)$ to be a path from the left-most vertex $0$ to the right-most vertex $n+2$ as follows. 
Let $\lambda_0(\pi)$ be the path from the vertex $0$ to the vertex $n+2$ passing through all lower-barred vertices   $\underline{i}\in\underline{[n+1]}$ in numerical order.  
Thus $\lambda_0(\pi)$ is the path along the lower boundary edges of $P(Q)$. 
Define $\lambda_1(\pi)$ as the piecewise linear path from $0$ to $n+2$ passing through the vertices
\[\left\{\begin{array}{ll} \underline{[n+1]}\cup \{\pi_1\},\quad &\textup{if $\pi_1\in\overline{[n+1]}$;}\\[5pt]
\underline{[n+1]}\setminus\{\pi_1\}, &\textup{if $\pi_1\in\underline{[n+1]}$,}\\
\end{array}\right.\] 
maintaining the numerical order of the vertices visited.  
Repeating this process recursively, the final path $\lambda_{n+1}(\pi)$ passes from $0$ to $n+2$ through all upper-barred vertices $\overline{i} \in \overline{[n+1]}$. 
Thus $\lambda_{n+1}(\pi)$ is the path along the upper boundary edges of $P(Q)$.

\begin{definition}\cite{Reading}
\label{def etaR} Define a map $\eta_Q\colon S_{n+1}\to \{\textup{triangulations of $P(Q)$}\}$, $\pi\mapsto \eta_Q(\pi)$, where $\eta_Q(\pi)$ is the triangulation  (including the boundary edges) of $P(Q)$ that arises as the union of the paths $\lambda_0(\pi), \dots, \lambda_{n+1}(\pi)$.  
\end{definition}
\begin{remark}
 It is  shown in \cite{Reading} that $\eta_Q$ is surjective and that its fibers correspond to the congruence classes of a certain lattice congruence on the weak order on the symmetric group. The induced poset structure is a lattice called a \emph{Cambrian lattice} of type $\mathbb{A}$.
 More precisely, two triangulations are ordered $\calt\le \calt'$ if there exist permutations $\pi\le \pi'$ in the weak order such that $\eta_Q(\pi) = \calt$ and $\eta_Q(\pi')=\calt'$.
\end{remark}
\begin{example}
\label{exam:eta}
Let $Q$ be the quiver in Figure~\ref{fig:polygon6}, where $n=5$. 
Then $\overline{[6]}=\{\overline{1},\overline{2},\overline{3},\overline{5}\}$ and $\underline{[6]}=\{\underline{4},\underline{6}\}$. 
Let $\pi=
4 \, 5 \, 3 \, 1 \, 2 \, 6
\in S_{6}$, written in one-line notation. 

Then the paths $\lambda_i(\pi)$ described above are as follows. 

\begin{align*}
     \lambda_0(\pi) &= 0,\underline{4},\underline{6},7 
     \\
     \lambda_1(\pi) &= 0,\underline{6},7  && \text{delete $\pi(1)=\underline{4}$ from $\lambda_0(\pi)$} 
     \\
     \lambda_2(\pi) &= 0,\overline{5},\underline{6},7 && \text{add $\pi(2)=\overline{5}$ to  $\lambda_1(\pi)$} 
     \\
     \lambda_3(\pi) &= 0,\overline{3},\overline{5},\underline{6},7 && \text{add $\pi(3)=\overline{3}$ to  $\lambda_2(\pi)$} 
     \\
     \lambda_4(\pi) &= 0,\overline{1},\overline{3},\overline{5},\underline{6},7 && \text{add $\pi(4)=\overline{1}$ to   $\lambda_3(\pi)$}
     \\
     \lambda_5(\pi) &= 0,\overline{1},\overline{2},\overline{3},\overline{5},\underline{6},7 && \text{add  $\pi(5)=\overline{2}$ to  $\lambda_4(\pi)$} 
     \\
     \lambda_6(\pi) &= 0,\overline{1},\overline{2},\overline{3},\overline{5},7 && \text{delete $\pi(6)=\underline{6}$ from   $\lambda_5(\pi)$} 
\end{align*} 
The triangulation $\calt=\eta_Q(\pi)$ is given in Figure~\ref{fig:example_triangulation}. Note that the fiber of $\calt$ is 
\begin{equation*}
  \eta_Q^{-1}(\calt) = 
  \left\{  
453126, 
453162, 
453612, 
456312  
  \right\}.
\end{equation*}
\end{example}

\begin{figure}
\definecolor{mygreen}{cmyk}{0.5,0,0.5,0.5}
\definecolor{mylightblue}{cmyk}{0.24,0.06,0,0.19}
\def\myxscale{0.5}
\def\myyscale{0.6}

\def\myboundaryedge{blue}
\def\myedgesize{ultra thick}
\tikzset{->-/.style={decoration={
			markings,
			mark=at position #1 with {\arrow{stealth}}},postaction={decorate}}}

\begin{tikzpicture}[xscale=\myxscale,yscale=\myyscale]
\tikzstyle{every node}=[font=\footnotesize]
\node (0) at (0,0) [fill,circle,inner sep=0.5] {};
		\node (1) at (0.7,1.2) [opacity=1,fill,circle,inner sep=0.5] {};
		\node (2) at (1.8,1.8) [fill,circle,inner sep=0.5] {};
		\node (3) at (3,2) [fill,circle,inner sep=0.5]{}; 
		\node (4) at (4,-1.5) [fill,circle,inner sep=0.5] {};
		\node (5) at (5,1.8) [fill,circle,inner sep=0.5] {};
		\node (6) at (6,-1) [fill,circle,inner sep=0.5] {};
		\node (7) at (6.5,0) [fill,circle,inner sep=0.5] {};
		
		\draw (0) node[left] {$0$};
		\draw (1) node[above] {$1$};
		\draw (2) node[above] {$2$};
		\draw (3) node[above] {$3$};
		\draw (4) node[below] {$4$};
		\draw (5) node[above] {$5$};
		\draw (6) node[below] {$6$};
		\draw (7) node[right] {$7$};

\draw[->-=0.66,blue,\myedgesize] (0) to [bend right=20] (4);		
\draw[->-=0.7,blue,\myedgesize] (4) -- (6); 
\draw[->-=0.7, blue, \myedgesize] (6) -- (7);






\end{tikzpicture}
\begin{tikzpicture}[xscale=\myxscale,yscale=\myyscale]
\tikzstyle{every node}=[font=\footnotesize]
\node (0) at (0,0) [fill,circle,inner sep=0.5] {};
		\node (1) at (0.7,1.2) [opacity=1,fill,circle,inner sep=0.5] {};
		\node (2) at (1.8,1.8) [fill,circle,inner sep=0.5] {};
		\node (3) at (3,2) [fill,circle,inner sep=0.5]{}; 
		\node (4) at (4,-1.5) [fill,circle,inner sep=0.5] {};
		\node (5) at (5,1.8) [fill,circle,inner sep=0.5] {};
		\node (6) at (6,-1) [fill,circle,inner sep=0.5] {};
		\node (7) at (6.5,0) [fill,circle,inner sep=0.5] {};
		
		\draw (0) node[left] {$0$};
		\draw (1) node[above] {$1$};
		\draw (2) node[above] {$2$};
		\draw (3) node[above] {$3$};
		\draw (4) node[below] {4};
		\draw (5) node[above] {$5$};
		\draw (6) node[below] {$6$};
		\draw (7) node[right] {$7$};

\draw[->-=0.66,blue,\myedgesize] (0) to [bend right=20] (4);		
\draw[->-=0.7,blue,\myedgesize] (4) -- (6); 
\draw[->-=0.7, blue, \myedgesize] (6) -- (7);

\draw[->-=0.7,red,\myedgesize] (0) -- (6); 





\end{tikzpicture}
\begin{tikzpicture}[xscale=\myxscale,yscale=\myyscale]
\tikzstyle{every node}=[font=\footnotesize]
\node (0) at (0,0) [fill,circle,inner sep=0.5] {};
		\node (1) at (0.7,1.2) [opacity=1,fill,circle,inner sep=0.5] {};
		\node (2) at (1.8,1.8) [fill,circle,inner sep=0.5] {};
		\node (3) at (3,2) [fill,circle,inner sep=0.5]{}; 
		\node (4) at (4,-1.5) [fill,circle,inner sep=0.5] {};
		\node (5) at (5,1.8) [fill,circle,inner sep=0.5] {};
		\node (6) at (6,-1) [fill,circle,inner sep=0.5] {};
		\node (7) at (6.5,0) [fill,circle,inner sep=0.5] {};
		
		\draw (0) node[left] {$0$};
		\draw (1) node[above] {$1$};
		\draw (2) node[above] {$2$};
		\draw (3) node[above] {$3$};
		\draw (4) node[below] {${4}$};
		\draw (5) node[above] {$5$};
		\draw (6) node[below] {$6$};
		\draw (7) node[right] {$7$};

\draw[->-=0.66,blue,\myedgesize] (0) to [bend right=20] (4);		
\draw[->-=0.7,blue,\myedgesize] (4) -- (6); 
\draw[->-=0.7, blue, \myedgesize] (6) -- (7);

\draw[->-=0.7,red,\myedgesize] (0) -- (6); 

\draw[->-=0.7,mygreen,\myedgesize] (0) to [bend right=20] (5); 
\draw[->-=0.7,mygreen,\myedgesize] (5) -- (6); 




\end{tikzpicture}
\begin{tikzpicture}[xscale=\myxscale,yscale=\myyscale]
\tikzstyle{every node}=[font=\footnotesize]
\node (0) at (0,0) [fill,circle,inner sep=0.5] {};
		\node (1) at (0.7,1.2) [opacity=1,fill,circle,inner sep=0.5] {};
		\node (2) at (1.8,1.8) [fill,circle,inner sep=0.5] {};
		\node (3) at (3,2) [fill,circle,inner sep=0.5]{}; 
		\node (4) at (4,-1.5) [fill,circle,inner sep=0.5] {};
		\node (5) at (5,1.8) [fill,circle,inner sep=0.5] {};
		\node (6) at (6,-1) [fill,circle,inner sep=0.5] {};
		\node (7) at (6.5,0) [fill,circle,inner sep=0.5] {};
		
		\draw (0) node[left] {$0$};
		\draw (1) node[above] {$1$};
		\draw (2) node[above] {$2$};
		\draw (3) node[above] {$3$};
		\draw (4) node[below] {$4$};
		\draw (5) node[above] {$5$};
		\draw (6) node[below] {$6$};
		\draw (7) node[right] {$7$};

\draw[->-=0.66,blue,\myedgesize] (0) to [bend right=20] (4);		
\draw[->-=0.7,blue,\myedgesize] (4) -- (6); 
\draw[->-=0.7, blue, \myedgesize] (6) -- (7);

\draw[->-=0.7,red,\myedgesize] (0) -- (6); 

\draw[->-=0.7,mygreen,\myedgesize] (0) to [bend right=20] (5); 
\draw[->-=0.7,mygreen,\myedgesize] (5) -- (6); 

\draw[->-=0.7,orange,\myedgesize] (0) to [bend right=20] (3);
\draw[->-=0.7,orange,\myedgesize] (3) to [] (5);



\end{tikzpicture}
\begin{tikzpicture}[xscale=\myxscale,yscale=\myyscale]
\tikzstyle{every node}=[font=\footnotesize]
\node (0) at (0,0) [fill,circle,inner sep=0.5] {};
		\node (1) at (0.7,1.2) [opacity=1,fill,circle,inner sep=0.5] {};
		\node (2) at (1.8,1.8) [fill,circle,inner sep=0.5] {};
		\node (3) at (3,2) [fill,circle,inner sep=0.5]{}; 
		\node (4) at (4,-1.5) [fill,circle,inner sep=0.5] {};
		\node (5) at (5,1.8) [fill,circle,inner sep=0.5] {};
		\node (6) at (6,-1) [fill,circle,inner sep=0.5] {};
		\node (7) at (6.5,0) [fill,circle,inner sep=0.5] {};
		
		\draw (0) node[left] {$0$};
		\draw (1) node[above] {$1$};
		\draw (2) node[above] {$2$};
		\draw (3) node[above] {$3$};
		\draw (4) node[below] {$4$};
		\draw (5) node[above] {$5$};
		\draw (6) node[below] {$6$};
		\draw (7) node[right] {$7$};

\draw[->-=0.66,blue,\myedgesize] (0) to [bend right=20] (4);		
\draw[->-=0.7,blue,\myedgesize] (4) -- (6); 
\draw[->-=0.7, blue, \myedgesize] (6) -- (7);

\draw[->-=0.7,red,\myedgesize] (0) -- (6); 

\draw[->-=0.7,mygreen,\myedgesize] (0) to [bend right=20] (5); 
\draw[->-=0.7,mygreen,\myedgesize] (5) -- (6); 

\draw[->-=0.7,orange,\myedgesize] (0) to [bend right=20] (3);
\draw[->-=0.7,orange,\myedgesize] (3) to [] (5);

\draw[->-=0.7,mypink,\myedgesize] (1) to [bend right=20] (3);
\draw[->-=0.7,mypink,\myedgesize] (0) to [] (1);


\end{tikzpicture}
\begin{tikzpicture}[xscale=\myxscale,yscale=\myyscale]
\tikzstyle{every node}=[font=\footnotesize]
\node (0) at (0,0) [fill,circle,inner sep=0.5] {};
		\node (1) at (0.7,1.2) [opacity=1,fill,circle,inner sep=0.5] {};
		\node (2) at (1.8,1.8) [fill,circle,inner sep=0.5] {};
		\node (3) at (3,2) [fill,circle,inner sep=0.5]{}; 
		\node (4) at (4,-1.5) [fill,circle,inner sep=0.5] {};
		\node (5) at (5,1.8) [fill,circle,inner sep=0.5] {};
		\node (6) at (6,-1) [fill,circle,inner sep=0.5] {};
		\node (7) at (6.5,0) [fill,circle,inner sep=0.5] {};
		
		\draw (0) node[left] {$0$};
		\draw (1) node[above] {$1$};
		\draw (2) node[above] {$2$};
		\draw (3) node[above] {$3$};
		\draw (4) node[below] {$4$};
		\draw (5) node[above] {$5$};
		\draw (6) node[below] {$6$};
		\draw (7) node[right] {$7$};

\draw[->-=0.66,blue,\myedgesize] (0) to [bend right=20] (4);		
\draw[->-=0.7,blue,\myedgesize] (4) -- (6); 
\draw[->-=0.7, blue, \myedgesize] (6) -- (7);

\draw[->-=0.7,red,\myedgesize] (0) -- (6); 

\draw[->-=0.7,mygreen,\myedgesize] (0) to [bend right=20] (5); 
\draw[->-=0.7,mygreen,\myedgesize] (5) -- (6); 

\draw[->-=0.7,orange,\myedgesize] (0) to [bend right=20] (3);
\draw[->-=0.7,orange,\myedgesize] (3) to [] (5);

\draw[->-=0.7,mypink,\myedgesize] (1) to [bend right=20] (3);
\draw[->-=0.7,mypink,\myedgesize] (0) to [] (1);

\draw[->-=0.7,violet,\myedgesize] (1) -- (2); 
\draw[->-=0.7,violet,\myedgesize] (2) -- (3); 

\end{tikzpicture}
\begin{tikzpicture}[xscale=\myxscale,yscale=\myyscale]
\tikzstyle{every node}=[font=\footnotesize]
\node (0) at (0,0) [fill,circle,inner sep=0.5] {};
		\node (1) at (0.7,1.2) [opacity=1,fill,circle,inner sep=0.5] {};
		\node (2) at (1.8,1.8) [fill,circle,inner sep=0.5] {};
		\node (3) at (3,2) [fill,circle,inner sep=0.5]{}; 
		\node (4) at (4,-1.5) [fill,circle,inner sep=0.5] {};
		\node (5) at (5,1.8) [fill,circle,inner sep=0.5] {};
		\node (6) at (6,-1) [fill,circle,inner sep=0.5] {};
		\node (7) at (6.5,0) [fill,circle,inner sep=0.5] {};
		
		\draw (0) node[left] {$0$};
		\draw (1) node[above] {$1$};
		\draw (2) node[above] {$2$};
		\draw (3) node[above] {$3$};
		\draw (4) node[below] {$4$};
		\draw (5) node[above] {$5$};
		\draw (6) node[below] {$6$};
		\draw (7) node[right] {$7$};

\draw[->-=0.66,blue,\myedgesize] (0) to [bend right=20] (4);		
\draw[->-=0.7,blue,\myedgesize] (4) -- (6); 
\draw[->-=0.7, blue, \myedgesize] (6) -- (7);

\draw[->-=0.7,red,\myedgesize] (0) -- (6); 

\draw[->-=0.7,mygreen,\myedgesize] (0) to [bend right=20] (5); 
\draw[->-=0.7,mygreen,\myedgesize] (5) -- (6); 

\draw[->-=0.7,orange,\myedgesize] (0) to [bend right=20] (3);
\draw[->-=0.7,orange,\myedgesize] (3) to [] (5);

\draw[->-=0.7,mypink,\myedgesize] (1) to [bend right=20] (3);
\draw[->-=0.7,mypink,\myedgesize] (0) to [] (1);

\draw[->-=0.7,violet,\myedgesize] (1) -- (2); 
\draw[->-=0.7,violet,\myedgesize] (2) -- (3); 

\draw[->-=0.7,mylightblue,\myedgesize] (5) -- (7); 
\end{tikzpicture}
\caption{The paths $\lambda_i(453126)$ and triangulation for $\eta_Q(453126)$ from Example~\ref{exam:eta}}
\label{fig:example_triangulation}
\end{figure}
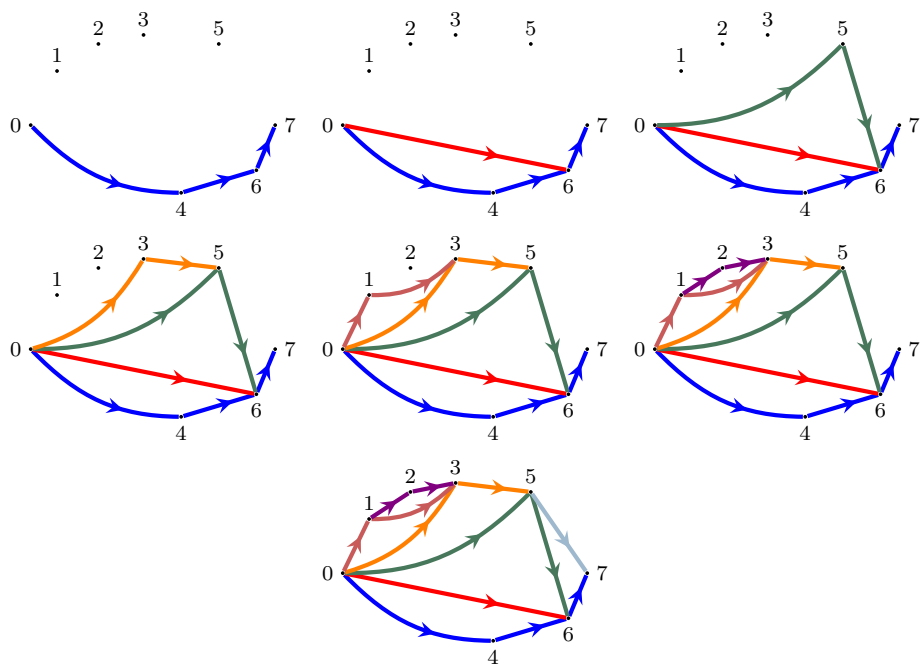


\section{Representations of quivers of type $\mathbb{A} $}\label{sect 3}
Let $\kb$ be an algebraically closed field, for example, $\kb=\mathbb{C}$. Given a quiver $Q$, we denote by $Q_0$ the set of its vertices and by $Q_1$ its set of arrows. For $\alpha\in Q_1$, let $s(\alpha)$ be the source of $\alpha$ and~$t(\alpha)$ be its target. A \emph{path} from $i$ to $j$ in $Q$ is a sequence of arrows $\za_1\za_2\ldots\za_\ell$ such that $s(\za_1)=i$, $t(\za_\ell)=j$, and $t(\za_h)=s(\za_{h+1})$, for all $1\le h\le \ell-1$. The integer $\ell $ is called the \emph{length} of the path. Paths of length zero are called \emph{constant paths} and are denoted by $e_i$, $i\in Q_0$.

 A \emph{representation} $M=(M_i,\varphi_\alpha)$ of $Q$ consists of a $\Bbbk$-vector space $M_i$, for each vertex $i\in Q_0$,  and a $\kb$-linear map $\varphi_\alpha\colon M_{s(\alpha)} \to M_{t(\alpha)}$, for each arrow $\alpha\in Q_1$. 
If each vector space $M_i$ is finite dimensional, we say that 
 $M$ is \emph{finite dimensional}, and the \emph{dimension vector} $\underline{\dim}\,M$  of $M$ is the vector $($dim$\,M_i)_{i\in Q_0}$ of the dimensions of the vector spaces.  
For example,  the representation in Figure~\ref{fig:plain_quiver} is a representation with dimension vector $(0,1,1,1,1,0,0)$ 
of the type $\mathbb{A}_7$ quiver in Figure~\ref{fig:polygon6}.
 Let $\rep Q$ denote the category of finite dimensional representations of $Q$ and
let $\ind Q$ denote a full
subcategory whose objects are one  representative of the
isoclass of each indecomposable representation.  The category $\rep Q$ is equivalent to the category of finitely generated modules over the path algebra $\kb Q$ of $Q$. The Auslander--Reiten quiver $\zG_{\rep Q}$ of $\rep Q$ has the isoclasses of indecomposable representations as vertices and irreducible morphisms as arrows. For more information about representations of quivers we refer to the textbooks~ \cite{ASS,S}.

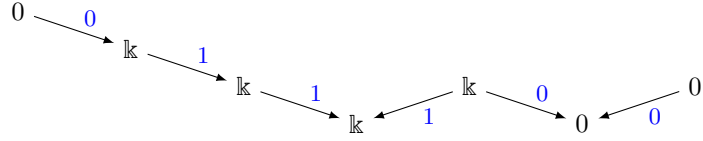
\begin{figure}[htbp]
\begin{center}
\begin{tikzpicture}[xscale=1.5,yscale=0.5,>=latex]
\def\posetedgecolor{blue}
\node(1) at (0,3) {$0$}; 
\node(2) at (1,2) {$\Bbbk$}; 
\node(3) at (2,1) {$\Bbbk$}; 
\node(4) at (3,0) {$\Bbbk$}; 
\node(5) at (4,1) {$\Bbbk$};   
\node(6) at (5,0) {$0$}; 
\node(7) at (6,1) {$0$}; 

\draw[->] (1) -- (2) node[\posetedgecolor,pos=0.7,above] {\small $0$};

\draw[->]  (2) -- (3) node[\posetedgecolor,pos=0.7,above] {\small $1$};

\draw[->] (3) -- (4) node[\posetedgecolor,pos=0.7,above] {\small $1$}; 

\draw[<-] (4) -- (5) node[\posetedgecolor,pos=0.7,below] {\small $1$}; 

\draw[->] (5) -- (6) node[\posetedgecolor,pos=0.7,above] {\small $0$}; 

\draw[<-] (6) -- (7) node[\posetedgecolor,pos=0.7,below] {\small $0$}; 
\end{tikzpicture}
\caption{A representation of the quiver in Figure~\ref{fig:polygon6}}
\label{fig:plain_quiver}
\end{center}
\end{figure}

From now on, let $Q$ be a Dynkin quiver of type $\mathbb{A}_{n+2}$, labeled as before  (see Figure \ref{fig:polygon6}).

In the remainder of this section, we recall the classification of indecomposable representations and irreducible morphisms in $\rep Q$, and give an interpretation of hooks and cohooks as boundary edges of the polygon $P(Q)$. 

\subsection{Indecomposable representations}
For each $1\le i\le j\le n+2$, let $M(i,j)$  denote the indecomposable representation supported on the vertices between $i$ and $j$. Thus $M(i,j)=(M_\ell,\varphi_\za)$ with
\[M_\ell=\left\{\begin{array}{ll}
\kb&\textup{if $i\le \ell\le j$;}\\
0&\textup{otherwise;}\end{array}\right.\]
and   $\varphi_\za=1$, whenever $M_{s(\za)}$ and $M_{t(\za)}$ are nonzero, and $\varphi_\za=0$, otherwise.
The representations $M(i,i)$ are \emph{simple} representations (containing no proper nonzero subrepresentations) and are denoted by $S(i)$, $i\in Q_0$.

We have the following proposition, which is a special case of Gabriel's theorem \cite{Gabriel}. 
\begin{prop}
 \label{prop gabriel} Let $Q$ be a Dynkin quiver of type $\mathbb{A}_{n+2}$.  
 Up to isomorphism, the indecomposable representations of $Q$ are precisely the representations $M(i,j)$ with $1\le i\le j \le n+2$.
\end{prop}

\subsection{Irreducible morphisms in terms of hooks and cohooks}
A morphism from a representation $(M_\ell, \phi_\alpha)$ to $(N_\ell, \psi_\alpha)$ is a sequence of linear maps $(f_\ell)$ such that for each arrow $\alpha\in Q_1$ the following diagram commutes.
\begin{center}
\begin{tikzpicture}[scale=0.60]
  \node (A) at (0,0) {$M_{s(\alpha)}$}; 
  \node(B) at (4,0) {$M_{t(\alpha)}$};
  \node (C) at (0,-3) {$N_{s(\alpha)}$}; 
  \node (D) at (4,-3) {$N_{t(\alpha)}$};
    \node at (2,-2.65) {$\psi_{\alpha}$};
    \node at (2,.35) {$\phi_{\alpha}$};
    \node at (-.65,-1.5) {$f_{s(\alpha)}$};
    \node at (4.75,-1.5) {$f_{t(\alpha)}$};
  \draw[->] (A.east)--(B.west);
  \draw[->] (A.south)--(C.north);
  \draw[->] (C.east)--(D.west);
  \draw[->] (B.south)--(D.north);
\end{tikzpicture}
\end{center}
In this subsection, we describe irreducible morphisms in type $\mathbb{A}$. 

Roughly speaking, an irreducible morphism between indecomposable representations is  a morphism that does not factor non-trivially. 
For a precise definition, see for example the textbooks~\cite{ASS,S}.
The work of Butler and Ringel from \cite{BR} shows  how each irreducible morphism  in a string algebra is determined by adding a hook or removing a cohook from an indecomposable representation to produce a new representation.
Recall that each irreducible morphism in $\rep(Q)$ is an arrow in the Auslander--Reiten quiver.
In Figure~\ref{fig:ar_quiver} we show an example.
Each (blue) $h$ and (red) $c$ indicates an arrow that corresponds to adding a hook and removing a cohook, respectively. 
In the following, we recall this construction in the special case of type $\mathbb{A}$.

\begin{definition}[\cite{BR}]
\label{def:hooks_cohooks}
Let $Q$ be a quiver of type $\mathbb{A}$  
and let $\alpha$ be an arrow in $Q$. 
The \emph{hook} of $\alpha$, denoted by hook$(\alpha)$ is given by the maximal path of $Q$ starting at $x=s(\alpha)$ that does not use $\alpha$. 
Note that this maximal path may consist of only one vertex  $x=s(\alpha)$  and no arrows. 

The \emph{cohook} of $\alpha$, denoted by cohook$(\alpha)$ is given by the maximal path of $Q$  ending at $y=t(\alpha)$ that does not use $\alpha$. Note that this maximal path may consist of only one vertex $y=t(\alpha)$ and no arrows. 

That is, hook$(\alpha)$ and  cohook$(\alpha)$ are of the form 
shown in Figure~\ref{fig:hook_cohook} (without the arrow $\alpha$).  

\begin{figure}
\[\xymatrix@!@R0pt@C0pt{& i\ar[ld]_\alpha\ar[rd]\\\cdot&& {i+1} \ar@{.}[rd]\\&&& {j-1}\ar[rd]\\&&&& j}
\qquad 
\xymatrix@!@R0pt@C0pt{&&&& j\ar[ld]\\&&& {j-1}\ar@{.}[ld]\\ \cdot \ar[rd]_\alpha&& {i+1} \ar[ld]\\& i}\]
\caption{The hook of the arrow $\alpha$, starting from $ i=s(\alpha)$ (left); the cohook of the arrow $\alpha$, ending at $ i=t(\alpha)$ (right)}
\label{fig:hook_cohook}
\end{figure}
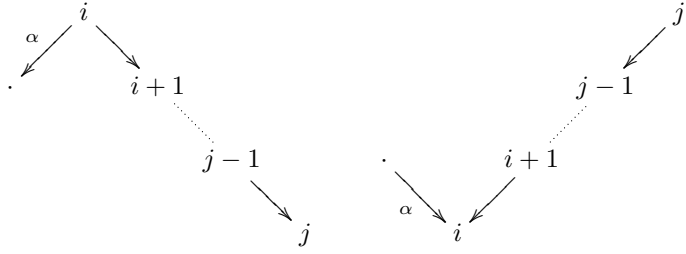

Denote by $H(\alpha)$, respectively $C(\alpha)$, the indecomposable representation of $Q$ supported on the hook, respectively cohook, of the arrow  $\alpha$, more precisely

$$H(\alpha)_x=\left\{\begin{array}{ll}\Bbbk &\textup{if $x$ is a vertex in hook$(\alpha)$;}\\
0&\textup{otherwise;}
\end{array}\right.
\qquad
C(\alpha)_x=\left\{\begin{array}{ll}\Bbbk &\textup{if $x$ is a vertex in cohook$(\alpha)$;}\\
0&\textup{otherwise.}
\end{array}\right.
$$

\end{definition}

\begin{remark}
It may happen that $H(\alpha)=C(\beta)$.  
For the quiver in Figure~\ref{fig:polygon6}, we have
$H(2\to 3)=C(1\to 2)=S(2)$
is the simple representation at vertex $2$, and 
$H(4 
\leftarrow 5 )=
C(6\leftarrow 7)=M(5,6)$
is the indecomposable representation supported at vertices $5$ and $6$. 
\end{remark}

Let $\alpha$ be an arrow in $Q_1$ whose target or source is either $i$ or $j$, and assume that $\alpha$ does not lie in the support of $M=M(i,j)$.
Adding a hook to $M$ is an operation which produces a new indecomposable representation $N$ whose support is
$\supp(M)\sqcup \supp(H(\za))$.
Similarly removing a cohook from a representation $N=M(i,\ell)$ or $M(h,j)$ is an operation which produces a new indecomposable representation $M$ whose support is $\supp(N)\setminus \supp(C(\za))$.
In the following proposition, 
we fix the representation $M$, 
and we either add a hook to $M$ to produce a ``bigger'' representation, or we remove a cohook from a ``bigger'' representation to obtain $M$.

\begin{prop}\label{prop:4possibilities}
Let $Q$ be a quiver of type $\mathbb{A}$, let $M=M(i,j)$, and let $N$ be an indecomposable representation satisfying $\supp(M)\subset \supp(N)$.
\begin{enumerate}
\item If $f: M\to N$ is an irreducible morphism then $N$ is obtained by adding a hook to~$M$. 
This can be done in at most two ways.
\begin{enumerate}
\item If $\za\colon j\ot (j+1)$ and $H(\za)=M(j+1,\ell)$ then $N= M(i,\ell)$.
\item If $\za: (i-1) \to i$ and $H(\za)=M(h,i-1)$ then $N=M(h,j)$.
\end{enumerate}
\item 
If $f: N\to M$ is an irreducible morphism then $M$ is obtained by removing a cohook from~$N$.
This can be done in at most two ways.
\begin{enumerate}
\item If $\za:j\to (j+1)$ and $C(\za)=M(j+1,\ell)$ then $N=M(i,\ell)$.
\item If $\za:(i-1)\ot i$ and $C(\za)=M(h,i-1)$ then $N=M(h,j)$.
\end{enumerate}
\end{enumerate}
Moreover all irreducible morphisms between indecomposable representations are of this form.
\end{prop}

\begin{example}
 In the example of Figure~\ref{fig:ar_quiver}, we can illustrate the four different types of irreducible morphisms as follows.
 In each example, $M$ is the representation with the smaller support.
\[ \begin{array}{lcclc}
{\rm (1)(a)} & \begin{smallmatrix}2\\3\\ 4
\end{smallmatrix}\to \begin{smallmatrix}
2 \quad \;\;  \\[-1mm] 
\; \; 3 \;\; 5 \\[-1mm] 
\quad \; \; 4 \;\; 6 
\end{smallmatrix}, \ \za\colon 5\to 4,\   H(\za)= 
\begin{smallmatrix}
 5\\6
\end{smallmatrix}
&\qquad&
 {\rm (2)(a)} & 
\begin{smallmatrix}
5\ 7\\6
\end{smallmatrix}
\to \begin{smallmatrix}5
\end{smallmatrix}, \ \za\colon 5\to 6,\   C(\za)= 
\begin{smallmatrix}
 7\\6
\end{smallmatrix}
\\ \\
{\rm (1)(b)} & \begin{smallmatrix}2\\3\\4
\end{smallmatrix}\to \begin{smallmatrix}1\\2\\3\\4
\end{smallmatrix}, \ \za\colon 1\to 2,\   H(\za)= 
\begin{smallmatrix}
 1
\end{smallmatrix}
&\qquad&
 {\rm (2)(b)} & 
\begin{smallmatrix}
 1\quad \\2\quad \\3\ 5\\4
\end{smallmatrix}
\to \begin{smallmatrix}5
\end{smallmatrix}, \ \za\colon 5\to 4,\   C(\za)= 
\begin{smallmatrix}
 1\\2\\3\\4\\ \\
\end{smallmatrix}
\end{array}\]
\end{example}

\subsection{Maximal increasing and decreasing paths}%

A path $i\to \ldots\to j$ in $Q$ is said to be \emph{increasing} if $i\le j$, and \emph{decreasing} if $i\ge j$. Note that the constant paths $e_i$ are both increasing and decreasing. 
A decreasing path is called a \emph{maximal decreasing path} if it is not a proper subpath of a decreasing path.
Similarly, 
an increasing path is called a \emph{maximal increasing path} if it is not a proper subpath of an increasing path.
Note that if $c$ is a maximal increasing path or 
a maximal decreasing path in a Dynkin quiver of type $\mathbb{A}$, then $c$ is a maximal path of $Q$ or a constant path. 
See Example~\ref{example after corollary A2}. 

The following lemma gives an interpretation of maximal 
increasing paths and maximal decreasing paths as boundary edges in the labeled polygon $P(Q)$. 
In the polygon $P(Q)$, let $\gamma(a,b)$ denote the line segment between vertices $a$ and $b$ for $0 \leq a < b \leq  n+2$. 

\begin{lemma}\label{lemma:maximal_paths_to_boundary_edges}
Let $Q$ be a Dynkin quiver of type $\mathbb{A}_{n+2}$. 

\begin{enumerate}
\item The map $(i\to\ldots\to j)\mapsto \gamma(i-1,j)$ is a bijection between  the maximal increasing paths in $Q$ and the lower boundary edges of $P(Q)$.

\item The map $( i\to\ldots\to  j)\mapsto \gamma(j-1,i)$ is   bijection between the maximal decreasing paths in $Q$ and the upper boundary edges of $P(Q)$.

\end{enumerate}
\end{lemma}
\begin{proof}
Recall that $\underline{[n+1]}$ is the set of vertices on the lower part of the boundary of $P(Q)$ and $\overline{[n+1]}$ is the set of vertices on the upper part of the boundary.

(1) 
Let $c$ be a maximal increasing path in $Q$. 
Since $Q$ is  of  type $\mathbb{A}$, either $c= e_i$ is constant or $c$ is a maximal path in $Q$. 

First, suppose $c=e_i$ is a maximal increasing path which is constant. Then either
$  (i+1) \to  i \to  (i-1)$ is a decreasing path in $Q$, or  
 $ i=1$ and is a sink, or 
 $ i=n+2$ and is a source. 
 Therefore $ i \in \underline{[n+1]}$ (if $i \neq n+2$)  
 and $ {i-1} \in \underline{[n+1]}$ (if $i \neq 1$).   
Then the line segment $\gamma(i-1,i)$ is a lower boundary edge.

Next, let $c$  be a non-constant maximal increasing path $   i \to  \dots \to  j$ in $Q$. 
Then $c$ is a maximal path in $Q$, so 
the vertex $ i$ is a source and the vertex $ j$ is a sink of $Q$. 
We want to show that the line segment $\gamma(i - 1, j)$ is a bottom boundary edge. Note that
the vertices $ i,\ldots, {j-1}$ are in
$\overline{[n+1]}$, by construction. 
Since $ i$ is a source, either $i=1$ 
or $ {i - 1}\in\underline{[n+1]}$. 
Since $ j$ is a sink, either $ j\in\underline{[n+1]}$ 
or $j=n+2$. 
Then the  vertices ${ i}, \ldots, { (j-1)}$ 
 form a maximal consecutive sequence of integers which label $j-i$ consecutive  vertices of the upper boundary of the polygon $P(Q)$, so $\gamma(i - 1, j)$ is a lower  boundary edge. 
It follows from the construction that this map is a bijection.

The proof of (2) is analogous to part (1). 
\end{proof}

\begin{corollary}
 A Dynkin quiver of type $\mathbb{A}_{n+2}$ has exactly $n+3$ maximal increasing and maximal decreasing paths.
\end{corollary}
\begin{proof}
 This follows from Lemma \ref{lemma:maximal_paths_to_boundary_edges}, since $P(Q)$ has exactly $n+3$ boundary edges.
\end{proof}

\begin{example}
\label{example after corollary A2}
Let $Q$ be the quiver $1\to 2$, so $n=0$. Then there is exactly one maximal increasing path $1\to 2$, and two maximal decreasing paths, namely the two constant paths $e_1$ and  $e_2$.

If $Q$ is the quiver from Figure \ref{fig:polygon6}, then $n=5$ and there are $n+3=8$ maximal increasing and maximal decreasing paths.
\end{example}

Recall that for each $\alpha\in Q_1$, the hook$(\alpha)$ and cohook$(\alpha)$ are defined to be maximal paths.
Therefore, in type $\mathbb{A}$ each hook and cohook is a maximal increasing or a maximal decreasing path.
We make this precise in the next lemma.

\begin{lemma}
\label{lemma:bij_hooks_cohooks_maximal_inc_dec_paths}
Let  $Q$ be a quiver of type $\mathbb{A}_{n+2}$. Then 
the set 
\[
\bigcup_{\alpha\in Q_1}\left\{H(\alpha),C(\alpha)\right\}
\] 
of all indecomposable representations supported on the hooks and cohooks 
is in bijection with 
\begin{itemize}
\item [\rm (a)]   the set $\{S(1), \dots,S( n+2) \}$ of all simple representations of $Q$, if $Q$ is the linear orientation of $\mathbb{A}_{n+2}$ (all arrows pointing in the same direction);
\item [\rm (b)] the set of maximal increasing paths and maximal decreasing paths, for all other choices of~$Q$.
\end{itemize}
\end{lemma}
\begin{proof}
We have noted earlier that every hook and every cohook is either a maximal increasing path or a maximal decreasing path. 
To prove (a) it suffices to notice that for a linearly oriented quiver, say, $1\to 2\to\ldots\to n+2$, the maximal decreasing paths are the $n+2$ constant paths,  and the unique maximal increasing path from $1$ to $n+2$ is neither a hook nor a cohook.

To prove (b) it remains to show that if $Q$ is not the linear orientation then every maximal increasing path and every maximal decreasing path is a hook or cohook.
Suppose $c$ is a maximal increasing path in $Q$. Since $Q$ is of type $\mathbb{A}$, either 
 $c$ is constant 
 or $c$ is a (non-constant) maximal path in $Q$. 
 
 First, suppose $c=e_i$ is a maximal increasing path which is constant. 
 Then either $ (i-1) \leftarrow  i \leftarrow  (i+1)$ is a decreasing path in $Q$, or 
$i=1$ and is a sink, 
or $i=n+2$ and  is a source.
In the first situation, $c$ is 
the hook of the arrow $(i-1) \leftarrow i$ in $Q$ 
(and also 
the cohook of the arrow $i \leftarrow (i+1)$ in $Q$). 
If $i=1$  is a sink, then $c$ is the cohook of the arrow $ 1 \leftarrow  2$ in $Q$. 
If $i=n+2$ is a source, then~$c$ is the hook of the arrow $ (n+1)\leftarrow (n+2)$ in $Q$.

Next, suppose c= $ i \to \dots \to  j$ is a non-constant maximal increasing path in $Q$. 
Then, as noted earlier, $c$ must be a maximal path in $Q$, so the vertex 
$ i$ is a source and the vertex $ j$ is a sink of $Q$. 
By assumption, the quiver $Q$ is not the linear orientation $ 1 \to  2 \to \dots \to  {n+2}$. 
Thus if $i>1$ we have an arrow $\alpha\colon(i-1) \leftarrow i$ in $Q$, since $i$ is a source. 
In this case,  
$c$ is the hook of the arrow $\alpha$, as illustrated in the left picture of  Figure~\ref{fig:hook_cohook}.  
On the other hand, if $i=1$ then $j<n+2$ and we have an arrow $\zb\colon{j} \leftarrow (j+1)$ in $Q$, since $j$ is a sink. In this case,  
$c$ is the cohook of the arrow $\zb$ (not illustrated). 

The proof that a maximal decreasing path corresponds to a hook or cohook is similar. 
\end{proof}

\begin{corollary}\label{cor 38}
Let $Q$ be a quiver of type $\mathbb{A}_{n+2}$. 
Then the set 
of all representations supported on the hooks and cohooks 
is in bijection with 

\begin{itemize}
\item [\rm (a) ] the set of all boundary edges of the polygon $P(Q)$  except for the long boundary edge $\gamma(0,n+2)$, if $Q$ is the linear orientation;
\item [\rm (b) ] the set of all boundary edges of the polygon $P(Q)$, for all other choices of $Q$.
\end{itemize}
\end{corollary}
\begin{proof}
This follows from 
Lemmas~\ref{lemma:bij_hooks_cohooks_maximal_inc_dec_paths}
 and \ref{lemma:maximal_paths_to_boundary_edges}.
\end{proof}

 
\section{A  geometric model for $\rep Q$}\label{sect 4}
In this section, we construct a category of line segments in the polygon $P(Q)$ and prove that it is equivalent to the category $\ind Q$ of indecomposable representations of $Q$. The construction is illustrated in an example in Figures~\ref{fig:ar_quiver} and \ref{fig:ar_quiver_polygon}. 

\begin{figure}
\adjustbox{scale=0.8, center}{
\begin{tikzcd}
[row sep=-0mm,column sep=tiny]
& P_7= 
\begin{matrix}
7 \\[-1mm]
6 
\end{matrix}  
\ar[dr, "h"{mylightblue}] && 
\begin{matrix}
5 \\[-1mm]
4 
\end{matrix}
\ar[dr, "h"{mylightblue}] && 3 \ar[dr, "h"{mylightblue}] && 
2  
\ar[dr, "h"{mylightblue}] && 
1=I_1\\ 
P_6=6 \ar[ur, "h"{blue}] \ar[dr, "h"{mylightblue}] && 
\begin{matrix}\,\,\,\, 5 \,\, 7 \\4 \,\, 6 \end{matrix} 
\ar[ur, "c"{red,swap}] \ar[dr, "h"{mylightblue}] &&  \begin{matrix}3 \; 5 \\4 \end{matrix} 
\ar[ur, "c"{red,swap}] \ar[dr, "h"{mylightblue}] && 
\begin{matrix}2 \\3 \end{matrix} \ar[ur, "c"{red,swap}] \ar[dr, "h"{mylightblue}] && 
\begin{matrix}1 \\2 \end{matrix} \ar[ur, "c"{red,swap}]=I_2 \\
& P_5= \begin{matrix}5 \\4 \,\, 6 \end{matrix} 
\ar[ur, "h"{blue}] \ar[dr, "h"{mylightblue}] && 
\begin{matrix}3 \; 5 \; 7 \\ 4 \; 6 \end{matrix} 
\ar[ur, "c"{red,swap}] \ar[dr, "h"{mylightblue}] && 
\begin{matrix}2 \quad  \\ \, \, \, \, \,   3 \, \, 5 \\ \quad 4 \end{matrix}
\ar[ur, "c"{red,swap}] \ar[dr, "h"{mylightblue}] && 
\begin{matrix}
1 \\[-1mm]
2 \\[-1mm]
3 
\end{matrix} = I_3 
\ar[ur, "c"{red,swap}]  \\
P_4=4 \ar[ur, "h"{blue}] \ar[dr, "h"{mylightblue}] && 
\begin{matrix}3 \;\; 5 \\ \;\;\; \; 4 \;\; 6 \end{matrix} \ar[ur, "h"{blue}]\ar[dr, "h"{mylightblue}] && 
\begin{matrix}2 \quad \quad \\\quad 3 \; \; 5 \; \; 7 \\\quad  4 \; \; 6 \end{matrix} 
\ar[ur, "c"{red,swap}] \ar[dr, "h"{mylightblue}] && 
\begin{matrix}
1 \quad \quad \\[-1mm]
2 \quad \\[-1mm]
\quad 3 \; \; 5 \\[-1mm]
\quad  4 
\end{matrix} = I_4 
\ar[ur, "c"{red,swap}] \ar[dr, "c"{mypink,swap}] \\
& P_3=\begin{matrix}3 \\4 \end{matrix} 
\ar[ur, "h"{blue}] \ar[dr, "h"{mylightblue}] && 
\begin{matrix}
2 \quad  \\[-1mm]
\quad 3 \; \; 5 \\[-1mm]
\quad \quad 4 \;\;  6 
\end{matrix} 
\ar[ur, "h"{blue}] \ar[dr, "h"{mylightblue}] && 
\begin{matrix}
1 \quad \quad \;\; \\[-1mm]
2 \quad \;\; \\[-1mm]
\quad \quad 3 \; \; 5 \;\; 7 \\[-1mm]
\quad \quad   4 \;\; 6 \end{matrix} 
\ar[ur, "c"{red,swap}] \ar[dr, "c"{mypink,swap}] && 
5 = I_5\\
&& 
P_2=\begin{matrix}
2\\[-1mm]
3 \\[-1mm]
4 
\end{matrix} 
\ar[ur, "h"{blue}] \ar[dr, "h"{mylightblue}] && 
\begin{matrix}
1 \quad \quad \;\;  \\[-1mm]
2 \quad \;\;  \\[-1mm] 
\; \; 3 \;\; 5 \\[-1mm] 
\quad \; \; 4 \;\; 6 
\end{matrix} 
\ar[ur, "h"{blue}] \ar[dr, "c"{mypink,swap}] && 
\begin{matrix}5 \, \, 7 \\6 \end{matrix} =I_6 \ar[ur, "c"{red,swap}] \ar[dr, "c"{mypink,swap}]\\
&&& 
P_1=\begin{matrix}
1 \\[-1mm]  
2 \\[-1mm]  
3 \\[-1mm]  
4 
\end{matrix} 
\ar[ur, "h"{blue}]  && \begin{matrix}5 \\6 \end{matrix} \ar[ur, "h"{blue}] && 7 = I_7
\end{tikzcd}
}
\caption{Auslander--Reiten quiver $\zG_{\rep Q}$ of rep $Q$ for the quiver $Q$ in  Figure~\ref{fig:polygon6}. 
}
\label{fig:ar_quiver}

\flushleft
\definecolor{mygreen}{cmyk}{0.5,0,0.5,0.5}
\def\pentascale{0.27}
\def\myyellow{white}
\def\mygray{gray}
\def\myboundaryedge{blue}
\def\myedgecolor{mygreen}
\def\myedgesize{ultra thick}
\tikzset{->-/.style={decoration={
			markings,
			mark=at position #1 with {\arrow{stealth}}},postaction={decorate}}}

\newcommand\mypentagon
{
	\begin{scope}[opacity=0.8,scale=\pentascale]
		\node (0) at (0,0) [fill,circle,inner sep=0.5] {};
		\node (1) at (0.7,1.2) [opacity=1,fill,circle,inner sep=0.5] {};
		\node (2) at (1.8,1.8) [fill,circle,inner sep=0.5] {};
		\node (3) at (3,2) [fill,circle,inner sep=0.5]{}; 
		\node (4) at (4,-1.5) [fill,circle,inner sep=0.5] {};
		\node (5) at (5,1.8) [fill,circle,inner sep=0.5] {};
		\node (6) at (6,-1) [fill,circle,inner sep=0.5] {};
		\node (7) at (6.5,0) [fill,circle,inner sep=0.5] {};
		\draw (0) -- (1) -- (2) -- (3) -- (5) -- (7)-- (6) -- (4) to [bend left=20] (0);
		
		\coordinate (northsource) at (7,2);
		\coordinate (northtarget) at (8.5,3.5);
		\coordinate (southsource) at (7,-2);
		\coordinate (southtarget) at (8.5,-3.5);
		
		\draw (0) node[left] {$0$};
		\draw (1) node[above] {$1$};
		\draw (2) node[above] {$2$};
		\draw (3) node[above] {$3$};
		\draw (4) node[below] {$4$};
		\draw (5) node[above] {$5$};
		\draw (6) node[below] {$6$};
		\draw (7) node[right] {$7$};
	\end{scope}
}
\adjustbox{scale=0.8}{
\begin{tikzpicture}
\tikzstyle{every node}=[font=\tiny]
\matrix[outer sep=0pt,column sep=-10mm,row sep=-0mm,
ampersand replacement=\&]
{
	\&
	\mypentagon
	\draw[->-=0.66,\myboundaryedge,\myedgesize] (5) -- (7);
	\draw (5) node [green,fill,circle,inner sep=1pt] {};
	\draw (7) node [red,fill,circle,inner sep=1pt] {};
	\draw[->] (southsource) -- (southtarget) node [pos=0.2, right, mylightblue] {h};
	\&
	\&
	\mypentagon{\mygray}
	\draw[->-=0.7,\myboundaryedge,\myedgesize] (3) -- (5);
	\draw (3) node [green,fill,circle,inner sep=1pt] {};
	\draw (5) node [red,fill,circle,inner sep=1pt] {};
	\draw[->] (southsource) -- (southtarget) node [pos=0.2, right, mylightblue] {h};
	\&
	\&
	\mypentagon{\mygray}
	
	\draw[->-=1,\myboundaryedge,\myedgesize] (2) -- (3);
	\draw (2) node [green,fill,circle,inner sep=1pt] {};
	\draw (3) node [red,fill,circle,inner sep=1pt] {};
	\draw[->] (southsource) -- (southtarget) node [pos=0.2, right, mylightblue] {h};
	\&
	\&
	\mypentagon{\mygray}
	\draw[->-=1,\myboundaryedge,\myedgesize] (1) -- (2);
	\draw (1) node [green,fill,circle,inner sep=1pt] {};
	\draw (2) node [red,fill,circle,inner sep=1pt] {};
	\draw[->] (southsource) -- (southtarget) node [pos=0.2, right, mylightblue] {h};
	\&
	\&
	\mypentagon{\myyellow}
	\draw[->-=1,\myboundaryedge,\myedgesize] (0) -- (1);
	\draw (0) node [green,fill,circle,inner sep=1pt] {};
	\draw (1) node [red,fill,circle,inner sep=1pt] {};
	\\
	\mypentagon{\mygray}
	\draw[->-=0.66,\myedgecolor,\myedgesize] (5) -- (6);
	\draw (5) node [green,fill,circle,inner sep=1pt] {};
	\draw (6) node [red,fill,circle,inner sep=1pt] {};
	\draw[->] (northsource) -- (northtarget)  node[pos=0.7,left,blue] {h};
	\draw[->] (southsource) -- (southtarget) node [pos=0.2, right, mylightblue] {h};
	\&
	\&
	\mypentagon{\mygray} 
	
	\draw[->-=0.66,\myedgecolor,\myedgesize] (3) -- (7);
	\draw (3) node [green,fill,circle,inner sep=1pt] {};
	\draw (7) node [red,fill,circle,inner sep=1pt] {};
	\draw[->] (northsource) -- (northtarget) node[pos=0.2,right,red] {c};
	\draw[->] (southsource) -- (southtarget) node [pos=0.2, right, mylightblue] {h};
	\&
	\&
	\mypentagon{\mygray}
	
	\draw[->-=0.66,\myedgecolor,\myedgesize] (2) to [bend right=60] (5);
	\draw (2) node [green,fill,circle,inner sep=1pt] {};
	\draw (5) node [red,fill,circle,inner sep=1pt] {};
	\draw[->] (northsource) -- (northtarget) node[pos=0.2,right,red] {c};
	\draw[->] (southsource) -- (southtarget) node [pos=0.2, right, mylightblue] {h};
	\&
	\&
	\mypentagon{\mygray}
	
	\draw[->-=0.66,\myedgecolor,\myedgesize] (1) to [bend right=60] (3);
	\draw (1) node [green,fill,circle,inner sep=1pt] {};
	\draw (3) node [red,fill,circle,inner sep=1pt] {};
	\draw[->] (northsource) -- (northtarget) node[pos=0.2,right,red] {c};
	\draw[->] (southsource) -- (southtarget) node [pos=0.2, right, mylightblue] {h};
	\&
	\&
	\mypentagon{\myyellow}
	\draw[->-=0.66,\myedgecolor,\myedgesize] (0) to [bend right=60] (2);
	\draw (0) node [green,fill,circle,inner sep=1pt] {};
	\draw (2) node [red,fill,circle,inner sep=1pt] {};
	\draw[->] (northsource) -- (northtarget) node[pos=0.2,right,red] {c};
	\&
	\&
	\&
	\&
	\\
	\&
	\mypentagon{\myyellow}
	\draw[->-=0.66,\myedgecolor,\myedgesize] (3) to (6);
	\draw (3) node [green,fill,circle,inner sep=1pt] {};
	\draw (6) node [red,fill,circle,inner sep=1pt] {};
	\draw[->] (northsource) -- (northtarget) node[pos=0.7,left,blue] {h};
	\draw[->] (southsource) -- (southtarget) node [pos=0.2, right, mylightblue] {h};
	\&
	\&
	\mypentagon{\mygray}
	\draw[->-=0.66,\myedgecolor,\myedgesize] (2) to [bend right=30] (7);
	\draw (2) node [green,fill,circle,inner sep=1pt] {};
	\draw (7) node [red,fill,circle,inner sep=1pt] {};
	\draw[->] (northsource) -- (northtarget) node[pos=0.2,right,red] {c};
	\draw[->] (southsource) -- (southtarget) node [pos=0.2, right, mylightblue] {h};
	\&
	\&
	\mypentagon{\mygray}
	\draw[->-=0.66,\myedgecolor,\myedgesize] (1) to [bend right=60] (5);
	\draw (1) node [green,fill,circle,inner sep=1pt] {};
	\draw (5) node [red,fill,circle,inner sep=1pt] {};
	
	\draw[->] (northsource) -- (northtarget) node[pos=0.2,right,red] {c};
	\draw[->] (southsource) -- (southtarget) node [pos=0.2, right, mylightblue] {h};
	\&
	\&
	\mypentagon{\myyellow}
	\draw[->-=0.66,\myedgecolor,\myedgesize] (0) to (3);
	\draw (0) node [green,fill,circle,inner sep=1pt] {};
	\draw (3) node [red,fill,circle,inner sep=1pt] {};
	\draw[->] (northsource) -- (northtarget) node[pos=0.2,right,red] {c};
	\&
	\&
	\&
	\&
	
	\\
	\mypentagon{\myyellow}
	\draw[->-=0.66,\myedgecolor,\myedgesize] (3) to (4);
	\draw (3) node [green,fill,circle,inner sep=1pt] {};
	\draw (4) node [red,fill,circle,inner sep=1pt] {};
	\draw[->] (northsource) -- (northtarget) node[pos=0.7,left,blue] {h};
	\draw[->] (southsource) -- (southtarget) node [pos=0.2, right, mylightblue] {h};
	\&
	\&
	\mypentagon{\mygray}
	\draw[->-=0.66,\myedgecolor,\myedgesize] (2) to (6);
	\draw (2) node [green,fill,circle,inner sep=1pt] {};
	\draw (6) node [red,fill,circle,inner sep=1pt] {};
	\draw[->] (northsource) -- (northtarget) node[pos=0.7,left,blue] {h};
	\draw[->] (southsource) -- (southtarget) node [pos=0.2, right, mylightblue] {h};
	\&
	\&
	\mypentagon{\myyellow}
	\draw[->-=0.66,\myedgecolor,\myedgesize] (1) to (7);
	\draw (1) node [green,fill,circle,inner sep=1pt] {};
	\draw (7) node [red,fill,circle,inner sep=1pt] {};
	\draw[->] (northsource) -- (northtarget) node[pos=0.2,right,red] {c};
	\draw[->] (southsource) -- (southtarget) node [pos=0.2, right, mylightblue] {h};
	\&
	\&
	\mypentagon
	\draw[->-=0.66,\myedgecolor,\myedgesize] (0) to (5);
	\draw (0) node [green,fill,circle,inner sep=1pt] {};
	\draw (5) node [red,fill,circle,inner sep=1pt] {};
	\draw[->] (northsource) -- (northtarget) node[pos=0.2,right,red] {c};
	\draw[->] (southsource) -- (southtarget) node [pos=0.7, left, mypink] {c};
	\&
	\&
	\&
	\&
	\\
	\&
	\mypentagon{\myyellow}
	\draw[->-=0.66,\myedgecolor,\myedgesize] (2) to (4);
	\draw (2) node [green,fill,circle,inner sep=1pt] {};
	\draw (4) node [red,fill,circle,inner sep=1pt] {};
	\draw[->] (northsource) -- (northtarget) node[pos=0.7,left,blue] {h};
	\draw[->] (southsource) -- (southtarget) node [pos=0.2, right, mylightblue] {h};
	\&
	\&
	\mypentagon{\myyellow}
	\draw[->-=0.66,\myedgecolor,\myedgesize] (1) to (6);
	\draw (1) node [green,fill,circle,inner sep=1pt] {};
	\draw (6) node [red,fill,circle,inner sep=1pt] {};
	\draw[->] (northsource) -- (northtarget) node[pos=0.7,left,blue] {h};
	\draw[->] (southsource) -- (southtarget) node [pos=0.2, right, mylightblue] {h};
	\&
	\&
	\mypentagon{\myyellow}
	\draw[->-=0.66,\myedgecolor,\myedgesize] (0) to (7);
	\draw (0) node [green,fill,circle,inner sep=1pt] {};
	\draw (7) node [red,fill,circle,inner sep=1pt] {};
	\draw[->] (northsource) -- (northtarget) node[pos=0.2,right,red] {c};
	\draw[->] (southsource) -- (southtarget) node [pos=0.7, left, mypink] {c};
	\&
	\&
	\mypentagon
	\draw[->-=0.66,\myedgecolor,\myedgesize] (4) to (5);
	\draw (4) node [green,fill,circle,inner sep=1pt] {};
	\draw (5) node [red,fill,circle,inner sep=1pt] {};
	\&
	\&
	\&
	\&
	\\
	\&
	\&
	\mypentagon
	\draw[->-=0.66,\myedgecolor,\myedgesize] (1) to (4);
	\draw (1) node [green,fill,circle,inner sep=1pt] {};
	\draw (4) node [red,fill,circle,inner sep=1pt] {};
	\draw[->] (northsource) -- (northtarget) node[pos=0.7,left,blue] {h};
	\draw[->] (southsource) -- (southtarget) node [pos=0.2, right, mylightblue] {h};
	\&
	\&
	\mypentagon
	\draw[->-=0.66,\myedgecolor,\myedgesize] (0) to [bend left=10] (6);
	\draw (0) node [green,fill,circle,inner sep=1pt] {};
	\draw (6) node [red,fill,circle,inner sep=1pt] {};
	\draw[->] (northsource) -- (northtarget) node[pos=0.7,left,blue] {h};
	\draw[->] (southsource) -- (southtarget) node [pos=0.7, left, mypink] {c};
	\&
	\&
	\mypentagon
	\draw[->-=0.66,\myedgecolor,\myedgesize] (4) to [bend left=10] (7);
	\draw (4) node [green,fill,circle,inner sep=1pt] {};
	\draw (7) node [red,fill,circle,inner sep=1pt] {};
	\draw[->] (northsource) -- (northtarget) node[pos=0.2,right,red] {c};
	\draw[->] (southsource) -- (southtarget) node [pos=0.7, left, mypink] {c};
	\\
	\&
	\&
	\&
	\mypentagon
	\draw[->-=0.66,\myboundaryedge,\myedgesize] (0) to [bend right=20] (4);
	\draw (0) node [green,fill,circle,inner sep=1pt] {};
	\draw (4) node [red,fill,circle,inner sep=1pt] {};
	\draw[->] (northsource) -- (northtarget) node[pos=0.7,left,blue] {h};
	\&
	\&
	\mypentagon
	\draw[->-=0.66,\myboundaryedge,\myedgesize] (4) to (6);
	\draw (4) node [green,fill,circle,inner sep=1pt] {};
	\draw (6) node [red,fill,circle,inner sep=1pt] {};
	\draw[->] (northsource) -- (northtarget) node[pos=0.7,left,blue] {h};
	\&
	\&
	\mypentagon
	\draw[->-=1,\myboundaryedge,\myedgesize] (6) to (7);
	\draw (6) node [green,fill,circle,inner sep=1pt] {};
	\draw (7) node [red,fill,circle,inner sep=1pt] {};
	\\
};
\end{tikzpicture}
}
\caption{Translation quiver of the category $\cc$ for the quiver $Q$ in Figure~\ref{fig:polygon6}. }
\label{fig:ar_quiver_polygon}
\end{figure}

\subsection{The category  of line segments $\cc$}
In this subsection, we construct a category $\cc$ whose   objects are line segments in the polygon $P(Q)$ and whose morphisms are generated by pivots of line segments modulo mesh relations. 

\subsubsection{Translation quivers and mesh categories}
We start by reviewing the notions of translation quiver and mesh category, from \cite{R, H}.
Recall that for any quiver $\zG$, we denote the set of vertices by $\zG_0$ and the set of arrows by $\zG_1$.  A \emph{loop} is an arrow that starts and ends at the same vertex.

A {\em  translation
 quiver} $(\zG,\tau)$ is a quiver $\zG=(\zG_0,\zG_1)$ without loops
 together with an injective map $\tau\colon \zG_0'\to\zG_0$  (the {\em translation}) from a subset $\zG_0'$ of $\zG_0$ to $\zG_0$ such that, for all vertices $x\in\zG_0'$, $y\in \zG_0$, 
the number of arrows from $y \to x$ is equal to the number of arrows
 from $\tau x\to y$.
Given a  translation quiver $(\zG,\tau)$, a \emph{polarization of} $\zG$ is
 an injective map $\sigma:\zG_1'\to\zG_1$, where $\zG_1'$ is the set of all arrows $\za\colon y\to x$ 
 with $x \in \zG_0'$, such that 
$\sigma(\za)\colon\tau x\to y$  for every arrow $\za\colon y\to x\in \zG_1$.

From now on we assume that $\zG$ has no multiple arrows. In that case, there is a unique polarization of $\zG$.

The {\em path category } of a translation quiver $(\zG,\tau)$ is the category whose  objects are
the vertices $\zG_0$ of $\zG$, and, given $x,y\in\zG_0$, the $\kb$-vector space of
morphisms from $x$ to $y$ is given by the $\kb$-vector space with basis
the set of all paths from $x $ to $y$. The composition of morphisms is
induced from the usual composition of  paths. 
The {\em mesh ideal} in the path category of $\zG$ is the ideal
generated by the {\em mesh relations}
\begin{equation}\nonumber
m_x =\sum_{\za:y\to x} \sigma(\za) \za.
\end{equation}  
for all $x \in \zG_0'$

For example, in the Auslander--Reiten quiver in Figure \ref{fig:ar_quiver}, we have the commutativity relation
\[m_{\begin{smallmatrix}
 3\ 5\\\ 4\ 6
\end{smallmatrix}}
= (\begin{smallmatrix}
4
\end{smallmatrix}
\to 
\begin{smallmatrix}
 5\\4\ 6
\end{smallmatrix}
\to
\begin{smallmatrix}
 3\ 5\\\ 4\ 6
\end{smallmatrix})
+
(\begin{smallmatrix}
4
\end{smallmatrix}
\to 
\begin{smallmatrix}
 3\\4
\end{smallmatrix}
\to
\begin{smallmatrix}
 3\ 5\\\ 4\ 6
\end{smallmatrix})\]
and the zero relation
$ m_2=(\begin{smallmatrix}
3
\end{smallmatrix}
\to 
\begin{smallmatrix}
2\\3
\end{smallmatrix}
\to
\begin{smallmatrix}
2
\end{smallmatrix}).$

The {\em  mesh category } $\calm (\zG,\tau)$ of $(\zG,\tau)$ is the
quotient of the path 
category of $(\zG,\tau)$ by the mesh ideal.

\begin{example}
 If $Q$ is a quiver of Dynkin type  then its Auslander--Reiten quiver $\zG_{\rep Q}$ together with the Auslander--Reiten translation $\tau$ is a translation quiver, where $\zG_0'$ is the set of all non-projective indecomposable representations. In this case, the mesh category $\calm(\zG_{\rep Q},\tau)$ is equivalent to the category $\ind Q$, and the additive closure of $\calm(\zG_{\rep Q},\tau)$ is equivalent to $\rep Q$.
\end{example}

\subsubsection{Definition of $\cc$}\label{sect cat}

We define  the set $\cale$ of all line segments $\zg(i,j)$ of $P(Q)$ as follows
\begin{equation}
\label{eq cale} 
\cale =\{\zg(i,j)\mid 0\le i<j\le n+2\},\end{equation}
 where $\zg(i,j)$ denotes the line segment between the vertices $i$ and $j$. 
 Notice that $\cale$ contains all boundary edges and all diagonals of $P(Q)$. 
 Also notice that our line segments are oriented by our choice of labeling, and that $\zg(i,j)\in\cale $ if and only if $\zg(j,i)\notin \cale$.
 
 For any vertex $\ell$ of $P(Q)$, denote by $R(\ell)$ and by $R^{-1}(\ell)$, respectively, the clockwise and counterclockwise neighbor of $\ell$ on the boundary of $P(Q)$.

\begin{definition}\label{def cale}
Let $\zg(i,j)\in \cale$. The line segment $\zg(i,R^{-1}(j))$ is called a \emph{pivot} of $\zg(i,j)$ if it lies in the set $\cale$, that is, if $i<R^{-1}(j)$. Similarly, the line segment  $\zg(R^{-1}(i),j)$ is called a \emph{pivot} of $\zg(i,j)$ if it lies in the set $\cale$, that is, if $R^{-1}(i)<j$. 
 \end{definition}
Thus a pivot of a line segment $\zg$ is given by fixing one of the endpoints of $\zg$ while moving the other endpoint to its counterclockwise neighbor.

\begin{remark} Notice that a line segment $\zg(i,j)$ may have two, one or zero pivots. 
 In the example in Figure \ref{fig:ar_quiver_polygon},
 the line segment $\zg(3,4)$ has the two pivots $\zg(2,4)$ and $\zg(3,6)$. On the other hand, the line segment $\zg(0,4)$ has only one pivot $\zg(0,6)$, and the line segment $\zg(4,5)$ has no pivots at all.
\end{remark}

 We now define a translation quiver $(\ga,R)$. The vertices of the quiver are the line segments in 
 $\cale$. There is an arrow $\zg(i,j)\to\zg(i',j')$ if and only if $\zg(i',j')$ is a pivot of $\zg(i,j)$, thus if and only if
 $i'<j'$ and $(i',j')=(R^{-1}(i),j)$ or 
 $(i',j')=(i,R^{-1}(j))$. 
 Finally, the translation $R$ is defined by 
 \begin{equation}\label{eq 42new}
R(\zg(i,j))=\left\{
\begin{array}{cl}
\zg(R(i),R(j)) &\textup{if $R(i)<R(j)$};\\
0&\textup{otherwise.}
\end{array}\right.
\end{equation}
Thus $R$ acts on $\cale$ by rotation.

The following is the main definition of this section.

\begin{definition}\label{def cc}
 Let $\cc$ be the mesh category of the translation quiver $(\ga,R)$. We call $\cc$ the \emph{category of line segments} of $P(Q)$.
\end{definition}

\subsection{The functor $F\colon\cc\to\ind Q$}
In this subsection, we construct an equivalence of categories between $\cc$ and $\ind  Q$.

\begin{definition}
Let $F\colon\cc\to\textup{ind}\, Q$ be the functor defined as follows.
On {\em objects}, let \[F(\gamma(i,j))=M(i+1,j).\] 

To define $F$ on  {\em morphisms}, it suffices to define it on the pivots introduced in Definition~\ref{def cale}.
Define $F\left(\zg(i,j)\to \zg(R^{-1}(i),j)\right) $ to  be the irreducible morphism $M(i+1,j)\to M(R^{-1}(i) +1 , j)$ given by \begin{itemize}
\item [(a)] adding the hook corresponding to the boundary edge $\zg(R^{-1}(i),i)$, if $R^{-1}(i)<i$;
\item[(b)]  removing the cohook corresponding to the boundary edge $\zg(i,R^{-1}(i))$, if $i<R^{-1}(i)$.

\end{itemize}
Similarly, 
let $F\left(\zg(i,j)\to \zg(i,R^{-1}(j))\right) $ be the irreducible morphism $M(i+1,j)\to M(i+1,R^{-1}(j) )$ given by \begin{itemize}
\item [(a)] adding the hook corresponding to the boundary edge $\zg(j,R^{-1}(j))$, if $j<R^{-1}(j)$;
\item[(b)]  removing the cohook corresponding to the boundary edge $\zg(R^{-1}(j),j)$, if $R^{-1}(j)<j$.

\end{itemize}
\end{definition}

We are now ready for our first main result.

\begin{thm}\label{thm 1}
 The functor $F$ is an equivalence of categories
 \[F\colon \cc\to\ind Q.\]
 In particular,
 \begin{enumerate}
 \item $F$ induces an isomorphism of translation quivers 
 $\left(\ga, R \right) \to \left(\zG_{\rep Q}, \tau \right)$; 
\item  $F$ induces bijections
 \[\begin{array}{rcl}
 \{ \textup{line segments in $P(Q)$}\}&\to& \ind Q;\\
\{ \textup{pivots  in $P(Q)$}\} &\to& \{\textup{irreducible morphisms in $Q$}\};
\end{array}\]
\item the rotation $R$ corresponds to the Auslander--Reiten translation $\tau$ in the following sense
\[F\circ R= \tau \circ F;\]
\item $F$ is an exact functor with respect to the induced abelian structure on $\cc$.
\end{enumerate}
\end{thm}
\begin{proof}
 Proposition \ref{prop gabriel} implies that the indecomposable representations are of the form $M(i,j)$ with $1\le i\le j\le n+2$. On the other hand, the line segments $\zg(i,j)\in \cale$ are parametrized by $0\le i<j\le n+2$ by equation (\ref{eq cale}). Thus $F$ is a bijection between the objects of the categories $\cc$ and $\ind Q$
and thus a bijection between the vertices of the translation quivers.

Proposition \ref{prop:4possibilities} 
combined with Corollary \ref{cor 38} shows that $F$ is a bijection between pivots and irreducible morphisms, hence a bijection between the arrows of the translation quivers. Moreover, $f\colon\zg(i,j)\to\zg(i',j')$ is an arrow in $\ga$ if and only if $F(f)\colon M(i+1,j)\to M(i'+1,j') $ is an arrow in $\zG_{\rep Q}$. Thus $F$ is an isomorphism of quivers $\ga\to\zG_{\rep Q}$.

Next we check that the translations of the two quivers correspond to each other. Let $\zg(i,j)\in \cale$ and assume that $R^{-1}(\zg(i,j))\in \cale$ as well. Note that
\begin{equation}
\label{eq 45}
 F(R^{-1}(\zg(i,j))) = F(\zg(R^{-1}(i),R^{-1}(j))) = M(R^{-1}(i)+1,R^{-1}(j)).
 \end{equation}
 Since $R^{-1}(\zg(i,j))\in \cale$ then $\zg(i,j)$ has at least one pivot, see Definition~\ref{def cale}. Without loss of generality, we may assume that $\zg(i,j)\to\zg(R^{-1}(i),j)$ is a pivot. 
 Then $\zg(R^{-1}(i),j)\to \zg(R^{-1}(i),R^{-1}(j))$ is a pivot too.
 Under $F$, the composition of these two pivots is mapped to the composition of two irreducible morphisms
 \[M(i+1,j)\to M(R^{-1}(i) +1,j)\to M(R^{-1}(i)+1,R^{-1}(j)).\]
Since $i\ne R^{-1}(i)$ and $j\ne R^{-1}(j)$, it follows from the structure of the Auslander--Reiten quivers of type $\mathbb{A}$ that the last representation in this sequence is the inverse Auslander--Reiten translate of the first, thus
\[M(R^{-1}(i)+1,R^{-1}(j))=\tau^{-1} M(i+1,j)=\tau^{-1} F(\zg(i,j)).\]
Now Equation (\ref{eq 45}) yields $F(R^{-1}(\zg(i,j))) =\tau^{-1} F(\zg(i,j)).$ This proves (1),(2) and (3). 

Since both categories $\cc$ and $\ind Q$ are the mesh categories of their translation quivers $\ga$ and $\zG_{\rep Q}$, statement (1) implies that $F$ is an equivalence of categories. In particular, this equivalence induces the structure of an abelian category on $\cc$. With respect to this structure~$F$ is exact, since every equivalence between abelian categories is exact.
\end{proof}

\begin{remark}
 In \cite{CCS}, the authors defined a category of unoriented diagonals (no boundary edges) in an $(n+3)$-gon, whose morphisms are also generated by pivots modulo mesh relations. That construction yields the cluster category of type $\mathbb{A}_n$. Our construction here is different in the following sense. First, we also include boundary edges, and our quiver is of rank $n+2$. 
 Second, our line segments are oriented, which implies that certain pivots that are allowed in the cluster category are not allowed in the category $\cc$. This is the reason why $\cc$ is a hereditary abelian category, while the cluster category is triangulated and not abelian.
 
 We investigate the relation between the two constructions further in section \ref{sect 6}.
\end{remark}


\section{Stability function}\label{sect stab}
Stability conditions were introduced in \cite{Schofield, King, Rudakov, Bridgeland} and have important applications in algebraic geometry and representation theory. We recall the definition in the setting of the category $\rep Q$ of finite-dimensional representations of a quiver $Q$.

Let $K_0$ denote the Grothendieck group of $\rep Q$. For $M\in \rep Q$, we denote its class in $K_0$ by~$[M]$. Thus for $M,N\in\rep Q$, we have $[M]=[N]$ if and only if $M$ and $N$ have the same dimension vector. $K_0$ has a basis given by the classes of the simple representations $S(x)$, $x\in Q_0$. 

\begin{definition}
(1) A \emph{central charge} is a group homomorphism $Z\colon K_0\to \mathbb{C}$ such that for all nonzero representations $M$ the complex number $Z([M])$ lies in the strict right half plane\footnote{ \cite{Bridgeland} uses the strict \emph{upper} half plane} \[\mathbb{H}=\{r\,e^{i\pi\phi}\mid r>0\textup{ and }\textstyle -\frac{1}{2}<\phi <\frac{1}{2}\}.\]
 
(2) Given a central charge $Z$, we obtain an associated \emph{stability function} \[\phi(M)=\frac{1}{\pi}\, \textup{arg}( Z([M])).\]

(3) Given a stability function $\phi$, a nonzero representation $M$ is called \emph{$\phi$-stable} if  every nonzero proper subrepresentation $L\subsetneq M$ satisfies $\phi(L)<\phi(M)$.

\end{definition}

\begin{remark}
(1) Since $Z$ is a group homomorphism it is determined by its values on the classes of simple representations.

(2) If $M$ is $\phi$-stable then $M$ is indecomposable. Indeed, since direct summands are subrepresentations, this follows directly from the additivity of the central charge.
\end{remark}

It is natural to ask if there exist stability functions for which all indecomposable representations are stable. Reineke conjectured in \cite[Conjecture 7.1]{Reineke} that such a stability function exists for every Dynkin quiver. In the same paper, he gave a stability function for the linearly oriented type $\mathbb{A}$ quiver. Apruzzese and Igusa proved the conjecture for all type $\mathbb{A}$ quivers in \cite{AI}, and, very recently, Kinser described all such stability functions for type $\mathbb{A}$ in \cite{Kinser}, see also \cite{Qiu15, Qiu, QZ22}.

 We are going to show that our geometric model also provides such a stability function  for every quiver of Dynkin type $\mathbb{A}$. The stability function is given via the functor $F$ of Theorem \ref{thm 1} simply by the slope of the oriented line segment of an indecomposable representation.

From now on let $Q$ be a quiver of type $\mathbb{A}_{n+2}$ and let $P(Q)$ be the $(n+3)$-gon constructed in Section \ref{sect 2}. Let $\cale$ be the set of oriented line segments 
\[\cale=\{\zg(j,k)\mid 0\le j<k\le n+2\}\]
as in section~\ref{sect 4}. 
Recall that $\zg(j,k)$ denotes the oriented line segment  from vertex $j$ to vertex $k$ of the polygon $P(Q)$.
Since the vertices of $P(Q)$ are points in the plane, we can define
a map $\textup{vec}\colon\cale\to \mathbb{C}, \zg(j,k)\mapsto \vec{\zg}(j,k)$, where  $\vec{\zg}(j,k)$ is the complex number $re^{i\theta}$  given by 
the
vector 
with the same direction and magnitude as the oriented line segment 
from point $j$ to point $k$.

Let $F\colon \cale\to \ind Q$ be 
the bijection of Theorem \ref{thm 1} and denote its inverse by $G$. Furthermore, let $\vec{G}=\textup{vec}\circ G$ denote its composition with the map defined above. Thus, for a representation $M=M(j,k)$ we have 
\[\vec{G}(M)=\vec{\zg}(j-1,k)= r(M) \,e^{i\theta(M)},\] 
 where 
$r(M)$ is the length of the vector $\vec{G}(M)=\vec{\zg}(j-1,k)$ and $\theta(M) $ is the angle from the positive real axis to the vector $\vec{G}(M)$ in the complex plane. An example is given in Figure~\ref{fig a1}.

Then $\vec{G}$ induces a group homomorphism
$Z\colon K_0\to \mathbb{C}$
defined on the basis of simples by $Z([S(x)])=\vec{\zg}(x-1,x)$ and extended additively to all of $K_0$. By definition, we have 
\[Z([M(j,k)])= Z\left(\sum_{x=j}^k \,[S(x)]\right) = \sum_{x=j}^k \vec{\zg}(x-1,x)=\vec{\zg}(j-1,k),\]
where the last identity is the addition in $\mathbb{C}$.  Thus for every indecomposable representation $M$, we have 
\begin{equation}
 \label{eq A1} 
 Z([M])=\vec{G}(M)= r(M) \, e^{{i\theta(M)}}.
\end{equation}
Note that $Z$ maps every representation of $Q$ to the strict right half plane, and thus $Z$ is a central charge.  
The corresponding stability function $\phi$ is given by
 \[ \phi(M)=\frac{1}{\pi}\theta(M)\in\textstyle \left(-\frac{1}{2},\frac{1}{2}\right).\]

\begin{figure}
\begin{center}
\scalebox{0.5}{\LARGE 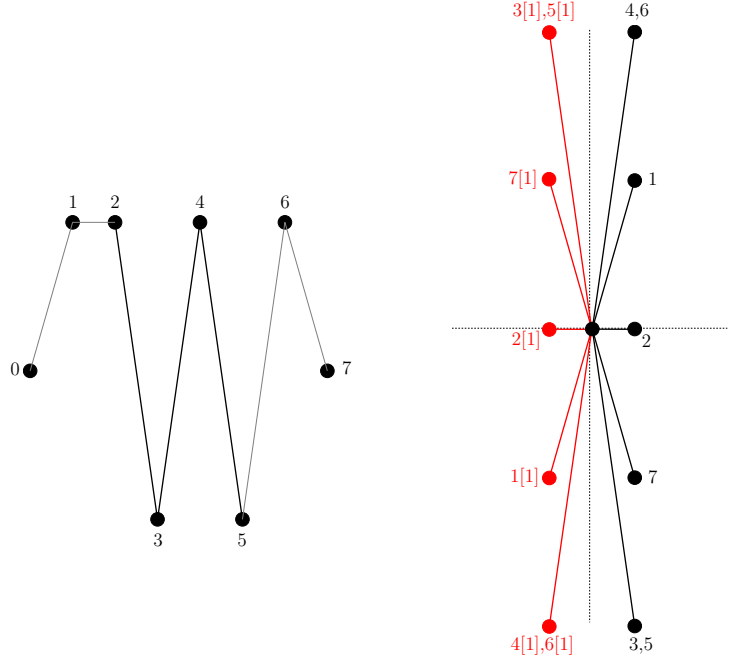}
 \caption{Line segments of the form $\zg(x,x+1)$ in the polygon $P(Q)$ (left); the corresponding central charges of the simple representations in black as well as their first shifts in the derived category in red (right)}
 \label{fig a1}
\end{center}
\end{figure}

Now, for every nonzero proper subrepresentation $L\subsetneq M$, the inclusion is a nonzero morphism in $\Hom(L,M)$. Thus the slope of $G(L)$ is strictly smaller than the slope of $G(M)$, by the same argument as in the proof of Lemma~\ref{lem 7}. Therefore $\phi(L)<\phi(M)$ and it follows that $M$ is $\phi$-stable. We have shown the following.
\begin{thm}
 \label{thm A1}
 Let $Q$ be a Dynkin quiver of type $A$. Then the function that associates to every indecomposable representation $M$ the angle of the corresponding oriented line segment $G(M)$ is a stability function for which every indecomposable representation of $Q$ is stable. 
\end{thm}

\begin{remark}
In \cite{Kinser}, Kinser characterizes the stability conditions in question by two sets of inequalities that are formulated using a decomposition of the quiver into horizontal and vertical segments. The horizontal segments in Kinser's quiver correspond to the lower boundary segments in our polygon in order from left to right, and his vertical segments correspond to our upper boundary segments also in order from left to right. Kinser's inequalities in \cite[Theorem 1.13]{Kinser} translate in our setting to the condition that the central charge (=slope) of the upper boundary segments is decreasing when moving along the upper boundary from left to right, and the central charge of the lower boundary segments is increasing.
\end{remark}

\subsection{A geometric model for the derived category}  The stability function  naturally allows for a geometric model for the derived category which we describe in this subsection. Geometric models for derived categories for arbitrary gentle algebras were given in \cite{OPS}, and our result here is equivalent to (a special case) of their construction. Note however, that the stability function is new.

Let $\cald=\cald^{b}(\rep Q)$ be the  derived category  of bounded complexes in $\rep Q$.  Since $\rep Q$ is a hereditary category, the indecomposable objects of $\cald$ are of the form $M[i]$, with $M$ an indecomposable representation of $Q$ and $i$ an integer, called the shift. Moreover if $M,N$ are indecomposable representations then \[\Hom_{\cald}(M[i],N[j])=\left\{\begin{array}{ll}\Hom_{\rep Q}(M,N) &\textup{if $j=i$;}\\
\Ext^{1}_{\rep Q}(M,N) &\textup{if $j=i+1$:}
\\ 0 &\textup{otherwise.}\end{array}\right.\]
In particular, the Auslander-Reiten quiver of $\cald$ is the translation quiver $\mathbb{Z}Q$, and in this quiver, each arrow corresponds either to (a shift of) an irreducible morphism in $\rep Q$ or to (a shift of) a short exact sequence of the form 
$0\to P(x)\to E\to I(y)\to 0$
with $P(x)$ an indecomposable projective, $I(y)$ an indecomposable injective, $E$ an indecomposable representation and there exists an arrow $y\to x$ in $Q$.

We know from Theorem~\ref{thm 1} that the irreducible morphisms in $\rep Q$ correspond to the pivots $\zg(i,j)\to \zg(R^{{-1}}(i), j)$ or $\zg(i,j)\to \zg(i, R^{-1}(j))$ of  oriented line segments in $\cale$, and these are defined only if $R^{{-1}}(i) < j$ or $i < R^{-1}(j)$, respectively. 
This condition can be reformulated by saying that the stability function $\phi$ evaluated on  the line segment after the pivot is still between $-\frac{1}{2}$ and $\frac{1}{2}$.

Now let us consider a pivot  $\zg(i,j)\to \zg(R^{{-1}}(i), j)$ that violates this condition, thus  $R^{{-1}}(i) > j$. 
This happens precisely when $i$ lies on the lower boundary (or $i=0$), all vertices $i+1,i+2,\ldots, j$ lie on the upper boundary, so that $R^{-1}(i)=k\ge j+1$ is the first vertex after $i$ on the lower boundary (possibly $k=n+2$). In particular, the quiver contains the following subquiver
\[ (i+1)\to\cdots\to j \to (j+1)\to\cdots\to k,\]
where  the vertex $(i+1)$ is a source and the vertex $k$ is a sink. In particular, the representation $F(\zg(i,j))=M(i+1,j)=I(j)$ is the injective at vertex $j$, and
 the reverse $\zg(j,R^{-1}(i))\in \cale$ of the pivot $\zg(R^{-1}(i),j)$ corresponds to the representation $F(\zg(j,R^{-1}(i)))=M(j+1,k)=P(j+1)$, which is the projective at vertex $j+1$. In particular, there is an arrow $j\to j+1$ and a short exact sequence
\[\xymatrix{0\ar[r]& P(j+1) \ar[r]^-{f}& M(i+1,k)\ar[r]^-{g}& I(j)\ar[r]& 0}.\]
Moreover, the line segment corresponding to the middle term of this sequence is $\zg(i,k)$ is a boundary segment on the lower boundary.

We define a generalized pivot to be a map $\zg(i,j)\to \zg(R^{{-1}}(i), j)$ or $\zg(i,j)\to \zg(i, R^{-1}(j))$ with $0\le i,j\le n+1$ but without any other condition on $i$ and $j$. 
Thus in the derived category a generalized pivot $\zg(i,j)\to \zg(R^{{-1}}(i), j)$ with $i<j$ and $R^{{-1}}(i) > j$ corresponds to a morphism $I(j)\to P(j+1)[1]$. 
Note that  $P(j+1)$  and its shift $P(j+1)[1]$ are represented by the same line segment, but with opposite orientations. 
Similarly,  a generalized pivot $\zg(i,j)\to \zg(i, R^{{-1}}(j))$ with $i<j$ and $i > R^{{-1}}(j)$ corresponds to a morphism $I(j)\to P(j-1)[1]$.

More generally, consider the following translation quiver $(\ga^{\mathbb{Z}}, R^{\mathbb{Z}})$. The  quiver $\ga^\mathbb{Z}$  has as vertices the pairs $(\zg,j)\in\cale\times\mathbb{Z}$,
 and  there is an arrow $(\zg,j)\to (\zg',j)$ if there is a pivot $\zg\to\zg'$, and there is an arrow $(\zg,j)\to (\overline{\zg'},j+1)$ if there is a generalized pivot $\zg\to \zg'$ but $\zg'\notin \cale$, where $\overline{\zg'}$ is the same line segment as $\zg'$ but in the opposite direction.
The translation $R^\mathbb{Z}$ is defined by 

\[R^\mathbb{Z}(\zg,j)=
\left\{
\begin{array}{ll} 
(R(\zg),j) &\textup{if $R(\zg)\in \cale$};\\[5 pt]
(\overline{R(\zg)},j-1) &\textup{if $R(\zg)\notin \cale,$}
\end{array}\right.\]
where $R$ is the rotation of a line segment defined in equation (\ref{eq 42new}).

Let $\cc^\mathbb{Z}$ be the mesh category of $\ga^\mathbb{Z}$. Then the functor $F$ of Theorem \ref{thm 1} induces an equivalence of categories
\[F^\mathbb{Z}\colon\cc^\mathbb{Z}\to\ind\cald\] that maps the pair $(\zg,j)$ to the object $F(\zg)[j]$.

Our stability function $\phi$ on $\rep Q$ induces a stability function on the derived category which we also denote by $\phi$, see \cite[Proposition 5.3]{Bridgeland}. This stability function is again given by the angle of the corresponding line segment. More precisely, for every indecomposable representation $M$ we define the angle $\theta(M[j])$ and the stability function $\phi(M[j])$ of the $j$-th shift of $M$ in $\cald$ as 
\[\theta(M[j])=\theta(M)+j\pi
\qquad \textup{and} \qquad 
\phi(M[j])=\phi(M)+j.\]
Note that $\phi(M[j])=\frac{1}{\pi}\theta(M[j] )$.


\section{Maximal almost rigid representations}\label{sect 5}
In this section, we give a characterization of the representations that correspond to the triangulations of $P(Q)$ under the equivalence $F\colon\cc\to\ind Q$.  
It seems that  this class of representations has not been studied so far. 

Recall that for $A,B\in \rep Q$, the vector space $\Ext(B,A)$ can be identified with the space of all short exact sequences of the form $0\to A\to E\to B\to 0$. Given such a short exact sequence, we call $A,B$ the end terms and $E$ the middle term of the sequence.
 Recall further that a  representation is called \emph{basic} if it has no repeated direct summands.
\begin{definition}
\label{def:almost_rigid}
We say that a  quiver representation $T$ is \emph{almost rigid} if it is basic and satisfies the following: For each pair $A,B$ of indecomposable summands of $T$, either $\Ext(A,B)=0$ or $\Ext(A,B)=\Bbbk$ and it is generated by a short exact sequence of the form $0 \to B \to E
\to A \to 0$ whose middle term $E$ is indecomposable.

We say that a  representation $T$  is \emph{maximal almost rigid} if,  for every nonzero representation $M$, the representation $T\oplus M$ is not almost rigid.

Let 
$\mar$ be the set of all maximal almost rigid representations of $Q$.
\end{definition}

\begin{example}
Let  $Q$ be the quiver $1\to 2$ of type $\mathbb{A}_2$. Up to isomorphism, there are precisely three indecomposable  representations of $Q$, namely $1,2$ and $\begin{smallmatrix}
1\\2
\end{smallmatrix}$. 
The representation $T'=1\oplus 2$ is almost rigid since the extension $\begin{smallmatrix}
1\\2
\end{smallmatrix}$ is indecomposable.
The representation $T=1\oplus\begin{smallmatrix}
1\\2
\end{smallmatrix}\oplus 2 $
is the only maximal almost rigid representation (up to isomorphism). 
\end{example}

\begin{example}
Let $Q$ be the quiver $1\to 2\to 3$. Then the representation $\begin{smallmatrix}
1\\2
\end{smallmatrix}\oplus \begin{smallmatrix}
2\\3
\end{smallmatrix}$ is not almost rigid, because there is an extension
\[0\to \begin{smallmatrix}
2\\3
\end{smallmatrix} \to
\begin{smallmatrix}
1\\2\\3
\end{smallmatrix} \oplus 
\begin{smallmatrix}
2
\end{smallmatrix}\to
\begin{smallmatrix}
1\\2
\end{smallmatrix} \to0 \]
whose middle term is not indecomposable.
\end{example}

\begin{remark}
At first sight, the definition of maximal almost rigid representations does not seem very natural from a representation theoretical perspective. However, at least for the path algebras of a Dynkin type $\mathbb{A}$ quiver, our Theorem~\ref{thm 2}  shows that these representations are important from a combinatorial perspective, since they do correspond to the triangulations of the polygon. 

It would be interesting to see if this correlation extends to other types of algebras beyond the path algebras of type $\mathbb{A}$. As we shall see in Example~\ref{ex D}, this does not seem to be the case for path algebras of Dynkin type $\mathbb{D}$. 
The reason for this difference is that an indecomposable representation in type $\mathbb{D}$ may contain a vertex with an incoming arrow and two outgoing arrows such that the composition of the incoming arrow with either one of the outgoing arrows is nonzero. 

Therefore, in order to find a more general class of algebras for which the maximal almost rigid representations may play a similar combinatorial role as in type $\mathbb{A}$, we need to look for algebras whose indecomposable representations are locally of the same form as those of type $\mathbb{A}$. 
These algebras are known as gentle algebras. 
They are given by quivers with relations (see Example~\ref{ex:gentlealgebra}). 
A combinatorial model for the module category of a gentle algebra has been found in~\cite{BCS21} by Baur and Coelho Sim\~{o}es, and a combinatorial model for the derived category of a gentle algebra in~\cite{OPS} by Opper, Plamondon and Schroll. 
It would be interesting to see if the maximal almost rigid representations (or complexes) of a gentle algebra realize combinatorially notable configurations in this model.
\end{remark}

In order to prove Theorem~\ref{thm 2}, we first need to recall the structure of extensions in type $\mathbb{A}$.
The following result is well-known. For a proof see for example \cite[Chapter 3]{S}.

\begin{prop}
\label{prop:ses} Let $Q$ be a quiver of type $\mathbb{A}_{n+2}$, and let $A,B$ be indecomposable representations of $Q$.
\begin{itemize}
\item [\rm (a)] There exists a non-split  short exact sequence 
with end terms $A,B$ and \emph{indecomposable} middle term $E$ if and only if
the relative position of $A,B,E$ in the Auslander--Reiten quiver defines a rectangle with one point missing as on the left of Figure~\ref{fig 54}.
\item [\rm (b)] There exists a non-split  short exact sequence 
with endterms $A,B$ and \emph{decomposable} middle term $E$ if and only if
$E$ has two indecomposable summands $E_1, E_2$, and the relative position of $A,B,E_1,E_2$ in the Auslander--Reiten quiver defines a rectangle  as on the right of Figure~\ref{fig 54}.
\end{itemize}
\end{prop}

\begin{figure}
 \[ \xymatrix@C15pt@R8pt{
 &&&&E\ar[rd]\ar@{<-}[ld]&&&&
 &&&&E_1\ar[rd]\ar@{<-}[ld]&&&&
 \\
&&&\cdot&&\cdot&&&
&&&\cdot&&\cdot&&&
\\ 
&&\cdot\ar@{.}[ru]\ar@{<-}[ld]&&&&\cdot\ar@{->}[rd]\ar@{.}[lu]&&
&&\cdot\ar@{<-}[ld]\ar@{.}[ru]&&&&\cdot\ar@{->}[rd]\ar@{.}[lu]&&
\\
&A\ar[rd]&&&&&&B\ar@{<-}[ld]& 
&A\ar[rd]&&&&&&B\ar@{<-}[ld]& 
\\
&&\cdot&&\cdot&&\cdot&&
&&\cdot&&&&\cdot&&
\\
&&&\cdot\ar@{.}[lu]\ar@{->}[ru]&&\cdot\ar@{<-}[lu]\ar@{.}[ru]&&&
&&&\cdot\ar@{.}[lu]\ar@{->}[rd]&&\cdot\ar@{<-}[ld]\ar@{.}[ru]&&&
\\
&&&&&&&&
 &&&&E_2
 }\]

\caption{Short exact sequences $0\to A\to E\to B\to 0$ seen in the Auslander--Reiten quiver. On the left, the middle term $E$ is indecomposable, on the right $E=E_1\oplus E_2$.}
\label{fig 54}

\end{figure}
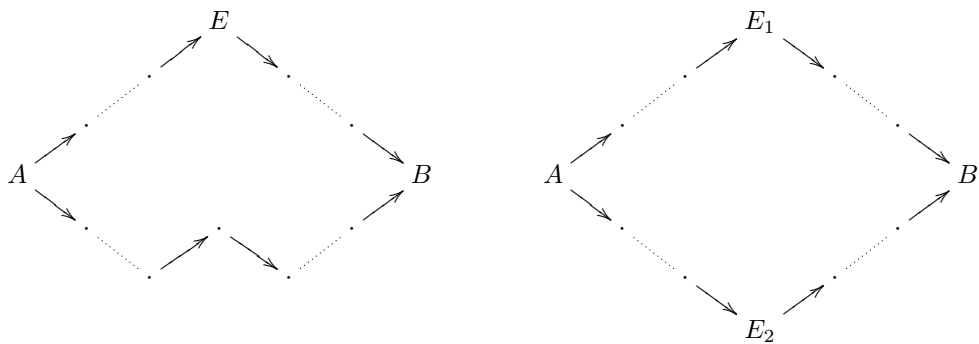

\begin{example}
Consider $P_2, I_4,I_6 $ and $I_7$ in Figure~\ref{fig:ar_quiver}. The rectangle defined by $P_2$ and $I_4$ lies completely within the AR-quiver and therefore there is an extension 
$0\to P_2\to M(1,4)\oplus M(2,5)\to I_4\to0$.

The rectangle defined by $P_2$ and $I_6$ is missing precisely one point, thus there is an extension 
$0\to P_2\to M(2,7)\to I_6\to0$ with indecomposable middle term.

On the other hand, the rectangle defined by $P_2$ and $I_7$ is missing three points, and hence $\Ext(I_7,P_2)=0$.

\end{example}

\begin{example}
Figures~ \ref{fig:crossingdiagonals:indecomposable_extension} and \ref{fig:crossingdiagonals:decomposable_extension} illustrate short exact sequences in terms of line segments. 
It may happen that two line segments share an endpoint but there is no extension between them. See Figure~\ref{fig:crossingdiagonals:tilting}. 
\end{example}

\begin{figure}
    \centering
\definecolor{mygreen}{cmyk}{0.5,0,0.5,0.5}
\def\pentascale{0.27}
\def\myyellow{white}
\def\mygray{gray}
\def\myboundaryedge{blue}
\def\myedgecolor{mygreen}
\def\myedgesize{ultra thick}
\tikzset{->-/.style={decoration={
			markings,
			mark=at position #1 with {\arrow{stealth}}},postaction={decorate}}}
\def\myxscale{1.4}
\def\myyscale{1.4}

\newcommand\mypentagon
{
	\begin{scope}[opacity=0.8,scale=\pentascale]
		\node (0) at (0,0) [fill,circle,inner sep=0.5] {};
		\node (1) at (0.7,1.2) [opacity=1,fill,circle,inner sep=0.5] {};
		\node (2) at (1.8,1.8) [fill,circle,inner sep=0.5] {};
		\node (3) at (3,2) [fill,circle,inner sep=0.5]{}; 
		\node (4) at (4,-1.5) [fill,circle,inner sep=0.5] {};
		\node (5) at (5,1.8) [fill,circle,inner sep=0.5] {};
		\node (6) at (6,-1) [fill,circle,inner sep=0.5] {};
		\node (7) at (6.5,0) [fill,circle,inner sep=0.5] {};
		
		\coordinate (northsource) at (7,2);
		\coordinate (northtarget) at (8.5,3.5);
		\coordinate (southsource) at (7,-2);
		\coordinate (southtarget) at (8.5,-3.5);
		
		\draw (0) node[left] {$0$};
		\draw (1) node[above] {$1$};
		\draw (2) node[above] {$2$};
		\draw (3) node[above] {$3$};
		\draw (4) node[below] {$4$};
		\draw (5) node[above] {$5$};
		\draw (6) node[below] {$6$};
		\draw (7) node[right] {$7$};
	\end{scope}
}

\begin{tikzpicture}[xscale=\myxscale,yscale=\myyscale] 
\tikzstyle{every node}=[font=\tiny]
	\mypentagon{\mygray}
	\draw[->-=0.40,\myedgecolor,\myedgesize] (2) -- (4);
	\draw (2) node [green,fill,circle,inner sep=1pt] {};
	
	\draw[->-=0.86,\myedgecolor,\myedgesize] (1) -- (2);
	\draw (1) node [green,fill,circle,inner sep=1pt] {};

\end{tikzpicture}
\begin{tikzpicture}[xscale=\myxscale,yscale=\myyscale] 
\tikzstyle{every node}=[font=\tiny]
	\mypentagon
	\draw[->-=0.66,violet,\myedgesize] (1) -- (4);
	\draw (1) node [green,fill,circle,inner sep=1pt] {};
	\draw (4) node [red,fill,circle,inner sep=1pt] {};
\end{tikzpicture}
\caption{Line segments $\gamma(1,2)$ and $\gamma(2,4)$ share an endpoint at the vertex $2$, and their extension is the line segment $\zg(1,4)$. Applying $F$  yields the short exact sequence $0\to M(3,4)\to M(2,4)\to S(2)\to 0$.}
\label{fig:crossingdiagonals:indecomposable_extension}

\begin{tikzpicture}[xscale=\myxscale,yscale=\myyscale] 
\tikzstyle{every node}=[font=\tiny]
	\mypentagon{\mygray}
	\draw[->-=0.40,\myedgecolor,\myedgesize] (2) to [] (4);
	\draw (2) node [green,fill,circle,inner sep=1pt] {};
	\draw (4) node [red,fill,circle,inner sep=1pt] {};

	\draw[->-=0.86,\myedgecolor,\myedgesize] (0) to [bend right=40] (5);
	\draw (0) node [green,fill,circle,inner sep=1pt] {};
	\draw (5) node [red,fill,circle,inner sep=1pt] {};
\end{tikzpicture}
\begin{tikzpicture}[xscale=\myxscale,yscale=\myyscale] 
\tikzstyle{every node}=[font=\tiny]
	\mypentagon
	\draw[->-=0.66,violet,\myedgesize] (0) to [bend right=20] (4);
	\draw (0) node [green,fill,circle,inner sep=1pt] {};
	\draw (4) node [red,fill,circle,inner sep=1pt] {};
	
    \draw[->-=0.66,violet,\myedgesize] (2) to [bend right=60] (5);
	\draw (2) node [green,fill,circle,inner sep=1pt] {};
	\draw (5) node [red,fill,circle,inner sep=1pt] {};
\end{tikzpicture}

\caption{Diagonals $\gamma(2,4)$ and $\gamma(0,5)$ cross each other, and their extension is the direct sum of $\zg(0,4)$ and $\zg(2,5)$. Applying $F$ yields the short exact sequence $0\to M(3,4)\to M(1,4)\oplus M(3,5)\to M(1,5)\to 0$.}
\label{fig:crossingdiagonals:decomposable_extension}

\begin{tikzpicture}[xscale=\myxscale,yscale=\myyscale] 
\tikzstyle{every node}=[font=\tiny]
	\mypentagon{\mygray}
	
	\draw[->-=0.66,\myedgecolor,\myedgesize] (0) to [bend right=20] (4);
	\draw (0) node [green,fill,circle,inner sep=1pt] {};
	\draw (4) node [red,fill,circle,inner sep=1pt] {};

	\draw[->-=0.70,\myedgecolor,\myedgesize] (0) to [] (1);
	\draw (0) node [green,fill,circle,inner sep=1pt] {};
	\draw (1) node [red,fill,circle,inner sep=1pt] {};
	
	\draw[->-=0.60,\myedgecolor,\myedgesize] (2) to [] (4);
	\draw (2) node [green,fill,circle,inner sep=1pt] {};
	\draw (4) node [red,fill,circle,inner sep=1pt] {};
	
\draw[->-=0.66,\myedgecolor,\myedgesize] (2) to [bend right=30] (5);
	\draw (2) node [green,fill,circle,inner sep=1pt] {};
	\draw (5) node [red,fill,circle,inner sep=1pt] {};	
	
\draw[->-=0.66,\myedgecolor,\myedgesize] (2) to [bend right=30] (7);
	\draw (2) node [green,fill,circle,inner sep=1pt] {};
	\draw (7) node [red,fill,circle,inner sep=1pt] {};

\draw[->-=0.76,\myedgecolor,\myedgesize] (6) -- (7);
	\draw (6) node [green,fill,circle,inner sep=1pt] {};
	\draw (7) node [red,fill,circle,inner sep=1pt] {};	
	
\draw[->-=0.76,\myedgecolor,\myedgesize] (3) -- (5);
	\draw (3) node [green,fill,circle,inner sep=1pt] {};
	\draw (5) node [red,fill,circle,inner sep=1pt] {};		
\end{tikzpicture}
\caption{$\Ext^1(A,B)=0$ for all $A,B$ in  $\{S(1)$, $M(1,4)$, $M(3,4)$, $M(3,7)$, $M(3,5)$, $M(4,5)$, $S(7)\}$ which corresponds to the set of illustrated line segments.}
\label{fig:crossingdiagonals:tilting}
\end{figure}

We have the following realization of maximal almost rigid representations.

\begin{thm}\label{thm 2}
 $F$ induces a bijection also denoted by $F$
 \[F\colon\{\textup{triangulations of $P(Q)$}\} \to \mar. 
 \] 
\end{thm}
\begin{proof} Let $M(i,j),M(i',j')$ be two non-isomorphic indecomposable representations of $Q$. Using Theorem \ref{thm 1} and  Proposition \ref{prop:ses} (b), we see that there is a short exact sequence of the form 
\[0\to M(i,j)\to E_1\oplus E_2\to M(i',j')\to 0\]
if and only if
there exist positive integers $s,t$ such that $E_1=M(i,R^{-t}(j))$, $E_2=M(R^{-s}(i),j)$ and $M(i',j')=M(R^{-s}(i),R^{-t}(j))$.
In particular, $i\le R^{-t}(j)$ and $R^{-s}(i)\le j$ and thus the line segments $F^{-1} (M(i,j))=\zg(i-1,j)$ and $F^{-1}(M(i',j'))=\zg(R^{-s}(i)-1,R^{-t}(j))$ are crossing diagonals in~$P(Q)$.

This shows that the sum of two indecomposable representations is almost rigid if and only if the corresponding line segments in $P(Q)$ do not cross.
Consequently, maximal almost rigid representations correspond to maximal sets of noncrossing line segments, hence triangulations.
\end{proof}
\begin{corollary}
\label{cor:catalan}
	If $Q$ is of type $\mathbb{A}_{n+2}$, 
	every maximal almost rigid representation of $Q$ has exactly $2n+3$ summands.  
	The maximal almost rigid representations of $Q$ are counted by the Catalan number $\binom{2n+2}{n+1} \frac{1}{n+2}$. 
\end{corollary}
\begin{proof}
 The first statement follows from the theorem, because each triangulation of the $(n+3)$-gon $P(Q)$ has $n$ diagonals and $n+3$ boundary edges. The second statement holds because the number of triangulations of $P(Q)$ is given by the Catalan number.
\end{proof}

\section{Endomorphism algebras of maximal almost rigid representations}\label{sect 6}
Recall that a module $T$ is \emph{rigid} provided that $\Ext(T,T)=0$.
We say that $T$ is a \emph{maximal rigid module} if $T\oplus M$ is not rigid for any nonzero module $M$.
A maximal rigid module over the path algebra $\kb Q$ of a quiver $Q$ is called a {\em tilting module}, and the endomorphism algebras of tilting modules over path algebras are called \emph{tilted algebras}, see \cite{HR}. On the other hand, the endomorphism algebras of triangulations  of polygons are {\em cluster-tilted algebras}  \cite{CCS}. 
Both tilted algebras and cluster-tilted algebras have been studied extensively. 

It is therefore natural to study the endomorphism algebras of the maximal \emph{almost} rigid representations. The purpose of this section is to show that in type $\mathbb{A}_{n+2}$  these algebras are tilted algebras of type $\mathbb{A}_{2n+3}$. 
For more details on endomorphism algebras and tilted algebras, see textbooks~\cite[Chapter 8]{ASS} and \cite[Chapter 6]{S}.

To every quiver $Q$ of type $\mathbb{A}_{n+2}$ we associate a quiver $\Qbar$ of type $\mathbb{A}_{2n+3}$ by replacing each arrow $i\to (i+1)$ by a path of length two $i\to \left(\frac{2i+1}{2}\right)\to (i+1)$
 and each arrow $i\ot (i+1)$ by a path of length two $i\ot \left(\frac{2i+1}{2}\right)\ot (i+1)$, see Figure~\ref{fig 6} for an example. 
 Since $Q$ has $n+2$ vertices and $n+1$ arrows, $\Qbar$ has $2n+3$ vertices which are labeled by the half-integers $i\in \frac{1}{2}\mathbb{Z}$ with $1\le i\le n+2$. Note that Proposition~\ref{prop gabriel} implies that the indecomposable representations of $\Qbar$ are of the form $\Mbar(i,j) $ with $i\le j$ and $i,j\in \frac{1}{2}\mathbb{Z}$.

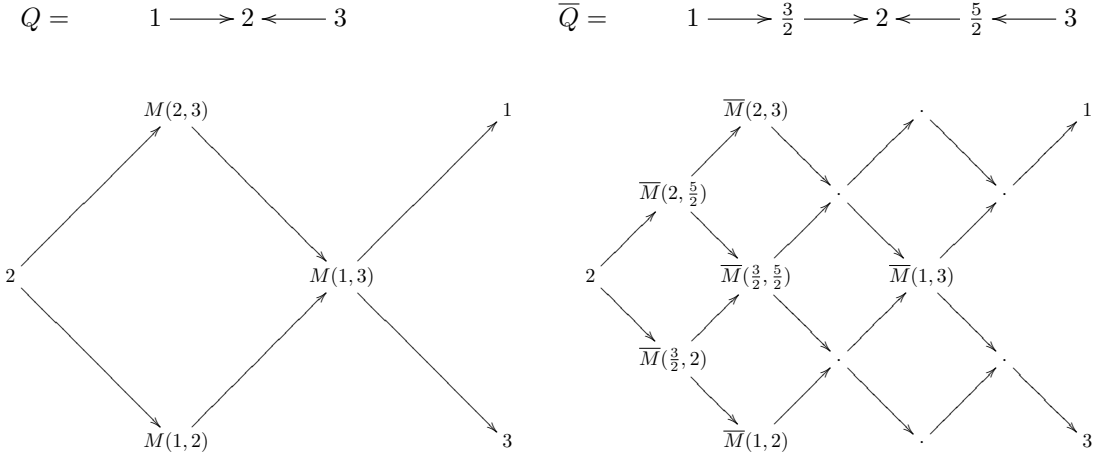
\begin{figure}
 
\begin{center}
 \[\xymatrix{Q=& 1\ar[r]&2&3\ar[l]&\qquad&\Qbar=&1\ar[r]&\frac{3}{2}\ar[r]&2&\frac{5}{2}\ar[l]&3\ar[l]}\]
 
 \[\scalebox{0.75}{\xymatrix@!@R0pt@C0pt{
 &&M(2,3)\ar[ddrr]&&&&1 
 &
&&\Mbar(2,3)\ar[dr]&&\cdot\ar[dr]&&1 
 \\ 
 &&&&&&&&\Mbar(2,\frac{5}{2})\ar[dr]\ar[ur]&&\cdot\ar[dr]\ar[ur]&&\cdot\ar[ur]\\ 
 2 \ar[uurr]\ar[ddrr] &&&&M(1,3)\ar[ddrr]\ar[uurr] &
 & 
 &2\ar[dr]\ar[ur]&&\Mbar(\frac{3}{2},\frac{5}{2})\ar[dr]\ar[ur]&&\Mbar(1,3)\ar[dr]\ar[ur]&&
 \\ 
&&& &&&&&\Mbar(\frac{3}{2},2)\ar[dr]\ar[ur]&&\cdot\ar[dr]\ar[ur]&&\cdot\ar[dr]
 \\ 
& &M(1,2)\ar[uurr]&&&&3
 &
 && \Mbar(1,2)\ar[ur]&&\cdot\ar[ur]&&3
 }}\]
 \caption{A quiver $Q$ and its Auslander--Reiten quiver on the left; the corresponding quiver $\Qbar$ and its Auslander--Reiten quiver on the right. The image of the functor $G$ are the representation $\Mbar(i,j)$ with $i$ and $j$ integers.}\label{fig 6}
\end{center}
\end{figure}
 Define a functor $G\colon \rep Q\to \rep \Qbar$ as follows. On indecomposable objects, we let $G(M(i,j))=\Mbar(i,j)$ and we extend it additively to all objects.
 If $f\colon M(i,j)\to M(i',j')$ is a morphism between indecomposables in  $\rep Q$, we let $G(f)=\fbar\colon\Mbar(i,j)\to \Mbar(i',j')$ defined by $\fbar_i=f_i$ if $i \in \mathbb{Z}$, and for the new vertices we let
 \[\fbar_{\frac{2i+1}{2}}=\left\{\begin{array}{ll}
 1&\textup{if $f_i=1$ and $f_{i+1}=1$}; \\
 0&\textup{otherwise.}
 \end{array}\right.\]
 
 For example, the irreducible morphism $2\to M(2,3)$ in $\rep Q$ in the example of Figure~\ref{fig 6} is mapped under $G$ to the composition of two irreducible morphisms $2\to\Mbar(2,\frac{5}{2})\to\Mbar(2,3)$.
 
\begin{lemma}
 \label{lem 6}
 The functor $G$ is full and faithful.
\end{lemma}
 
\begin{proof}
 Faithfulness is clear from the definition. To show that $G$ is full, take a nonzero morphism $\fbar\in \Hom_{\Qbar}(G(M(i,j)),G(M(i',j')))$, thus $\fbar\colon\Mbar(i,j)\to\Mbar(i',j')$. We want to show that $\fbar=G(f)$ for some $f\in\Hom_{Q}(M(i,j),M(i',j'))$.
 Let $\Ibar\subset\Qbar_0$ be the set of all vertices on which  both $\Mbar(i,j) $ and $\Mbar(i',j')$ are nonzero. Then $\fbar_x=1$, if $x\in \Ibar$, and $\fbar_x=0$, otherwise. Let $I\subset \Ibar$ be the subset of vertices from $Q$ and define $f\colon M(i,j)\to M(i',j')$ by $f_x=1$, if $x\in I$, and $f_x=0$, otherwise. Then~$f$ is a morphism in $\rep Q$ and $G(f)=\fbar$. 
\end{proof}

\begin{corollary}
 \label{cor 6}
 For all $T\in \rep Q$, we have $\End_{\rep Q}\, T\cong \End_{\rep \Qbar}\, G(T).$
\end{corollary}
\begin{proof}
 This follows since $G$ is full and faithful.
\end{proof}

 To state our next result, we recall that  the \emph{trivial extension} $C\ltimes E$ of an algebra $C$ by a $C$-bimodule $E$ is the algebra whose underlying vector space is the direct sum $C\oplus E$, and whose multiplication is defined as
$(c,e)(c',e')=(cc',ce'+ec')$. For more details, see for example the textbook~\cite[Chapter~6.2]{S}

We also need the notion of the \emph{adjacency quiver $Q(\calt)$ of a triangulation $\calt$ of a polygon.} The vertices of $Q(\calt)$ are given by the line segments in $\calt$. And there is an arrow $\zg\to \zg'$ in $Q(\calt)$ if the line segments $\zg,\zg'$ bound the same triangle in $\calt$  such that $\zg'$ follows $\zg$ when going along the boundary of the triangle in counterclockwise direction. For an example see Figure~\ref{fig 63}.
\begin{figure}
\begin{center}\small
\def\alphanum{\ifcase \xypolynode \or \!\!\!\!\zg(0,1) \or\!\!\!\! \zg(0,2)\or \ \zg(2,3)\!\!\!\! \or \zg(1,3)\!\!\!\!\or \or f\fi}
 \[
\xy/r4pc/: {\xypolygon4"A"{~={100}
 ~<<{@{}}~><{@{-}}
~>>{^{{\alphanum}}}
}},
\POS"A1" \ar@{-}_{\zg(1,2)\!\!} "A3",
\endxy
\quad \longrightarrow\quad
\xy/r4pc/: {\xypolygon8"A"{~={100}~<<{@{}}~><{@{-}}
}},
\POS"A1" \ar@{-}^{\!\!\!\!\zg(0,1)} "A3",
\POS"A3" \ar@{-}^{\!\!\!\!\zg(0,2)} "A5",
\POS"A5" \ar@{-}^{\zg(2,3)\!\!\!\!} "A7",
\POS"A7" \ar@{-}^{\zg(1,3)\!\!\!\!} "A1",
\POS"A1" \ar@{-}_{\zg(1,2)\!\!} "A5",
\endxy \qquad\quad\qquad\qquad \raisebox{2.5pc}{\xymatrix@R15pt@C15pt{\scriptstyle\zg(0,1)\ar[dd]&&\scriptstyle\zg(1,3)\ar[ld]\\&\scriptstyle\zg(1,2)\ar[lu]\ar[rd]\\ \scriptstyle\zg(0,2)\ar[ru]&&\scriptstyle\zg(2,3)\ar[uu]}}
\]
 \caption{Left: The embedding of a polygon $P(Q)$ with triangulation $\calt$ into the polygon $\Pbar$. 
 Right: The quiver $Q(\calt)$. 
 }\label{fig 63}
\end{center}
\end{figure}
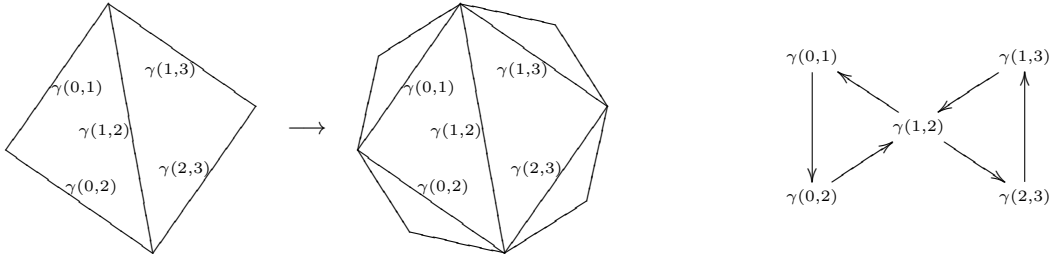

We are ready for the main theorem of this section.
\begin{thm}
 \label{thm 3}
 Let $Q$ be a quiver of type $\mathbb{A}_{n+2}$, and let $T$ be a maximal almost rigid representation of $Q$ with endomorphism algebra $C=\End_{\rep Q}\, T$. 
 Then
 
 (1) $C$ is a tilted algebra of type $\Qbar$.
 
 (2) The quiver of the cluster-tilted algebra 
 $B=C\ltimes\Ext^2_C(DC,C)$ 
 is the adjacency quiver $Q(\calt)$ of the triangulation $\calt=F^{-1}(T)$. Moreover, every arrow of $Q(\calt)$ lies in a (unique) 3-cycle.
 
 (3) The quiver $Q(C)$ of $C$ is obtained from $\calt$ by deleting the arrows of $Q(\calt)$ of the form $\gamma(i,j) \to \gamma(j,k)$ and $\gamma(j,k) \to \gamma(i,j)$, $i<j<k$. 
Moreover, every arrow of $Q(C)$ lies in a (unique) relation.
\end{thm}

\begin{example}\label{ex 6}
 Let $Q$ be the quiver of Figure~\ref{fig 6} and let $T$ be the maximal almost rigid representation $T=2\oplus M(1,2)\oplus M(2,3)\oplus 1\oplus 3$.
 The corresponding triangulation $\calt$ is shown left in Figure~\ref{fig 63}.
 Then $C=\End_{\rep Q}\,T$ is given by the quiver in Figure \ref{fig 61}(top)  bound by the relations  $\za\zb=0$ and $\zg\zd=0$. 
Its cluster-tilted algebra $B=C\ltimes\Ext^2_C(DC,C)$ 
is given by the quiver in Figure~\ref{fig 61}(bottom) bound by the relations $\za\zb=\zb\ze=\ze\za=0$ and $\zg\zd=\zd\zs=\zs\zg=0$.

\end{example}

\begin{proof}[Proof of Theorem~\ref{thm 3}] Let  $\Qbar$ be the quiver of type $\mathbb{A}_{2n+3}$ defined in the beginning of this section.
The cluster category $\cbar$ of $\Qbar$ can be realized in two ways. For one, $\cbar=\cald^b(\rep \Qbar)/\tau^{-1}[1]$ is an orbit category of the bounded derived category of representations of $\Qbar$, see \cite{BMRRT}. 
 On the other hand, $\cbar$ is the category of diagonals in a $(2n+6)$-gon $\Pbar$, see \cite{CCS}. Using the realization via the derived category, we get a functor
 $\widetilde{G}\colon\rep Q \to \cbar$, which maps a representation to its orbit in the cluster category. 
 This functor is not full, since, in the cluster category, there exist morphisms between the orbits of two representations that are not given by morphisms between the representations themselves. 
 
 On the other hand, using the realization via diagonals of $\Pbar$, we see that the composition $\widetilde{G}\circ F\colon \cc\to \cbar$ is a functor from the category of line segments (=diagonals and boundary edges)
 in the $(n+3)$-gon $P(Q)$ to the category of diagonals (not including boundary edges) in the  $(2n+6)$-gon $\Pbar$.
 We now explain this construction on the level of the polygons. We embed $P(Q)$ into $\Pbar$ by adding one vertex for each boundary edge; thus each boundary edge 
 of $P(Q)$ becomes an interior edge of $\Pbar$ which bounds a triangle whose other two sides are boundary edges, as follows, see Figure~\ref{fig 63} for an example.
 \[ \xymatrix@R=0pt{&&&&&\bullet\ar@{-}[rd]\\
 i\ar@{-}[rr]&&j&\mapsto&i\ar@{-}[rr]\ar@{-}[ru]&&j}\]

 The functor $\widetilde{G}\circ F$ induces a map $\calt\to \caltbar$ from triangulations of $P(Q)$ that include all boundary edges to triangulations of $\Pbar$ that do not include any boundary edges.  
 
 Now  let $T$ be a maximal almost rigid representation of $Q$. By Theorem~\ref{thm 2},  there exists a  triangulation $\calt$ of $P(Q)$ such that $T=F(\calt)$.  Let $C=\End_{\rep Q}\, T$ be its endomorphism algebra. Denote by $\caltbar$ the triangulation (without boundary edges) of $\Pbar$ given as the image of $\calt$ under $\widetilde{G}\circ F$. We have the following  commutative diagram. 
 
\[\xymatrix@C60pt@C50pt{\cc\ar[d]_F\ar[drr]^{\widetilde{G}\circ F}\\ \rep Q \ar@/_15pt/[rr]_{\widetilde{G}}\ar[r]^G &\rep \Qbar \ar[r]^\iota & \cbar}
\]
 
 Thus the object $\widetilde {G}(T)=\iota\circ G(T)$ in the cluster category is given by the triangulation $\caltbar$. Therefore $\widetilde{G}(T)$ is a cluster-tilting object in $\cbar$, by \cite{CCS}.  Let $B=\End_{\cbar}\, \widetilde {G}(T)$ denote its cluster-tilted algebra.
 
 Furthermore, we know that $\widetilde{G}(T)$ is induced from the representation	 $G(T)\in \rep \Qbar$, and therefore $G(T)$ is a tilting module over the path algebra $\kb \Qbar$ of $\Qbar$, by \cite{BMRRT,ABS2}. Its endomorphism algebra $\Cbar=\End_{\rep \Qbar}\,G(T)$ is a tilted algebra of type $\Qbar$ and its trivial extension $\Cbar\ltimes \Ext^2_{\Cbar}\,(D\Cbar, \Cbar)$ is the cluster-tilted algebra $B$, by \cite{ABS}. Corollary~\ref{cor 6} implies that $C\cong \Cbar$, so $C$ is tilted of type~$\Qbar$ and $B\cong C\ltimes\Ext^2_C(DC,C)$. 
 The quiver of $B$ is the quiver of the triangulation $\caltbar$, by \cite{CCS}, and thus it is also the quiver of the triangulation $\calt$. This shows (1) and (2).
 
 To show (3) we need to consider the morphisms in the category $\cc$. By definition, they are given by composition of pivots of the form $\zg(i,j)\to \zg(i,R^{-1}(j))$ or $\zg(i,j)\to \zg(R^{-1}(i),j)$. Given a triangle in $\calt$ with vertices $i<j<k$ in the labeling of $P(Q)$, 
 and such that traveling from $i$ to $j$ to $k$ to $i$ is going counterclockwise around the triangle, there exists a morphism $\zg(i,j)\to \zg(i,k)$ given by a sequence of pivots that fix the endpoint $i$, and a morphism $\zg(i,k)\to\zg(j,k)$  given by a sequence of pivots that fix the endpoint $k$. In the cluster category $\Cbar$ there also is a nonzero morphism  $\zg(j,k)\to \zg(j,i)$ given by a sequence of pivots fixing the endpoint $j$, 
 however, this morphism is zero in the category $\cc$ because the diagonal $\zg(j,i)$ is not in $\cale$, since $j>i$. 

Similarly, when the triangle $\calt$ has vertices $i<j<k$ in the labeling of $P(Q)$ such that traveling from $i$ to $j$ to $k$ is going clockwise around the triangle, there is a morphism $\gamma(i,k) \to \gamma(i,j)$ given by a sequence of pivots that fix $i$, and there is a morphism $\gamma(k,j) \to \gamma(i,k)$ given by a sequence of pivots that fix $k$. 
 \end{proof}
 \begin{figure}
\begin{center}
 \[\xymatrix@R0pt{Q(C)&
3\ar[r]^(0.35)\za&M(2,3)\ar[r]^(0.6)\zb &2&M(1,2)\ar[l]_(0.6)\zd& 1\ar[l]_(0.35)\zg}\]
 \[\xymatrix@R0pt{Q(B)&
3\ar[r]_(0.35)\za&M(2,3)\ar[r]_(0.6)\zb &2\ar@/_{10pt}/[ll]_\ze\ar@/^{10pt}/[rr]^\zs&M(1,2)\ar[l]^(0.6)\zd& 1\ar[l]^(0.35)\zg}\]
\caption{The quivers of $C$ and $B$ in Example \ref{ex 6}. To obtain $Q(C)$ from $Q(B)$ remove the arrows $\epsilon$ and $\sigma$ as described in part (3) of Theorem~\ref{thm 3}}
\label{fig 61}
\end{center} 
\end{figure}

\begin{remark}
 Not every tilted algebra of type $\mathbb{A}_{2n+3} $ is realizable as the endomorphism algebra of a maximal almost rigid representation of type $\mathbb{A}_{n+2}$. For example the tilted algebra given by the quiver $\xymatrix{1\ar[r]^\za&2\ar[r]^\zb&3&4\ar[l]&5\ar[l]}$ with relation $\za\zb=0$ is not, because not every arrow lies in a relation.
\end{remark}

\begin{remark}
 It would be interesting to see how the maximal almost rigid representations behave for other quivers or more generally for bound quiver algebras. The first example below shows that in Dynkin type $\mathbb{D}_4$ the number of summands in a maximal almost rigid representation is not always the same.
\end{remark}

\begin{example}\label{ex D}
 The $\mathbb{D}_4$ quiver $Q$ \[\xymatrix@R0pt{&&3\\1\ar[r]&2\ar[ru]\ar[rd]\\&&4}\]
  admits the following maximal almost rigid representations, one of which has 7 direct summands and the other has 9. 
 \[T=
P(1)\oplus 
P(2)\oplus
P(3)\oplus
P(4)\oplus\begin{smallmatrix}
 2\\4
\end{smallmatrix}\oplus
\begin{smallmatrix}
 2\\3
\end{smallmatrix}
\oplus
\begin{smallmatrix}
 1
\end{smallmatrix}
\qquad \textup{and}\qquad T'=
P(1)\oplus
P(3)\oplus
P(4)\oplus
\begin{smallmatrix}
 2\\4
\end{smallmatrix}\oplus
\begin{smallmatrix}
 2\\3
\end{smallmatrix}\oplus
\begin{smallmatrix}
1\\ 2\\4
\end{smallmatrix}\oplus
\begin{smallmatrix}
 1\\2\\3
\end{smallmatrix}
\oplus
\begin{smallmatrix}
 2
\end{smallmatrix}\oplus
\begin{smallmatrix}
 1
\end{smallmatrix}
\]
 The endomorphism algebra of $T$ is tilted of affine type $\widetilde{\mathbb{E}}_6=\Qbar$, which is encouraging for a possible generalization of the type $\mathbb{A}$ results.  However, the endomorphism algebra  of $T'$ does not seem to be a tilted algebra. 
\end{example}

\begin{example}
\label{ex:gentlealgebra}
Consider the $\mathbb{D}_4$ quiver 
\[\xymatrix@R0pt{&&3\\1\ar[r]^(0.5)\alpha&2\ar[ru]^(0.5)\beta\ar[rd]\\&&4}\]
bound by the relation $\alpha \beta=0$. The Auslander--Reiten quiver is illustrated in 
Figure~\ref{fig:gentlealgebra}. 

If an indecomposable representation $M$ has at most one incoming arrow and one outgoing arrow in the Auslander--Reiten quiver, then $M$ is a direct summand of every maximal almost rigid representation. 
Therefore, the five indecomposable representations 
written in bold 
in Figure~\ref{fig:gentlealgebra}  
 are direct summands of every maximal almost rigid representation. 

The maximal almost rigid representations are as follows. Each of them has $7$ direct summands.
\begin{align*}
\mathbf{\textcolor{violet}{3}}
\oplus 
\mathbf{\textcolor{violet}{4}} 
\oplus 
\textcolor{violet}{
\begin{matrix}\mathbf{2}\\[-1mm]\mathbf{3}\end{matrix}  
}
\oplus  
\textcolor{violet}{
\begin{matrix}\mathbf{1}\\[-1mm]\mathbf{2}\\[-1mm]\mathbf{4}\end{matrix} 
}
\oplus 
\mathbf{\textcolor{violet}{1}}
&
\oplus 
\begin{matrix}
2\\[-1mm]4
\end{matrix}
\oplus
\begin{matrix}2 \\[-1mm]3 \, 4 \end{matrix}
\\ 
\mathbf{\textcolor{violet}{3}}
\oplus 
\mathbf{\textcolor{violet}{4}} 
\oplus 
\textcolor{violet}{
\begin{matrix}\mathbf{2}\\[-1mm]\mathbf{3}\end{matrix}  
}
\oplus  
\textcolor{violet}{
\begin{matrix}\mathbf{1}\\[-1mm]\mathbf{2}\\[-1mm]\mathbf{4}\end{matrix} 
}
\oplus 
\mathbf{\textcolor{violet}{1}}
&
\oplus 
\begin{matrix}
2\\[-1mm]4
\end{matrix}
\oplus
2
\\ 
\mathbf{\textcolor{violet}{3}}
\oplus 
\mathbf{\textcolor{violet}{4}} 
\oplus 
\textcolor{violet}{
\begin{matrix}\mathbf{2}\\[-1mm]\mathbf{3}\end{matrix}  
}
\oplus  
\textcolor{violet}{
\begin{matrix}\mathbf{1}\\[-1mm]\mathbf{2}\\[-1mm]\mathbf{4}\end{matrix} 
}
\oplus 
\mathbf{\textcolor{violet}{1}}
&
\oplus 
2
\oplus 
\begin{matrix}
1\\[-1mm]2
\end{matrix}
\end{align*}
\end{example}
\begin{figure}
\begin{tikzcd}
& 
&
&
\textcolor{violet}
{\mathbf
{
P_1=
\begin{matrix}
\mathbf{1} \\[-1mm]
\mathbf{2} \\[-1mm]
\mathbf{4}
\end{matrix}
=I_4
}
}
 \ar[dr]
&& 
&
& 
\\ 
\textcolor{violet}{\mathbf{P_3=3}}
\ar[dr] 
&
& 
\text{rad}~P_1=
\begin{matrix}2\\[-1mm]4
\end{matrix} 
 \ar[ur] 
 \ar[dr] 
&
&  
\begin{matrix}1 \\[-1mm]2 \end{matrix} 
=I_2
 \ar[dr]
&
& 
\\
& P_2= \begin{matrix}2 \\[-1mm]3 \, 4 \end{matrix}
\ar[ur]
\ar[dr]
&
& 
2 
\ar[ur] 
&
& 
\textcolor{violet}{\mathbf{1=I_1}}
\\
\textcolor{violet}{\mathbf{P_4=4}} 
\ar[ur] 
&
& 
\textcolor{violet}
{
\begin{matrix}\mathbf{2} \\[-1mm] \mathbf{3}\end{matrix} 
=I_3
}
\ar[ur]
&
& 
\end{tikzcd}
\caption{Auslander--Reiten quiver for the bound quiver in  
Example~\ref{ex:gentlealgebra}
}
\label{fig:gentlealgebra}
\end{figure}
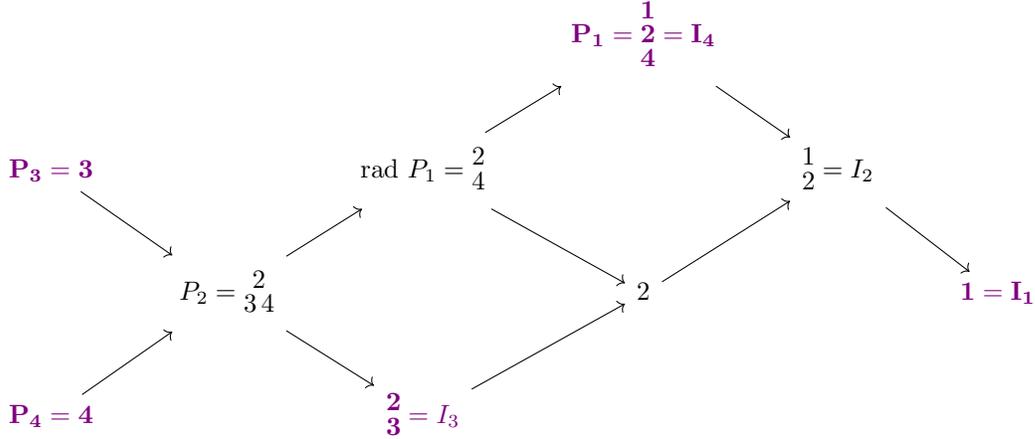
\section{Representation theoretic version of the Cambrian lattice and the $\eta$ map} \label{sect 7}
In this section, we come back to our initial motivation and give a new realization of the $\eta$ map in terms of maximal almost rigid representations.

 \subsection{Poset structure on $\mar$}
A \emph{flip} is a transformation of a triangulation $\calt$ that removes a diagonal $\gamma$ and replaces it with a (unique) different diagonal $\gamma'$ that, together with the remaining diagonals, forms a new triangulation $\calt'$.
Note that the two diagonals $\gamma$ and $\gamma'$ involved in such a flip must cross.

In~\cite{Reading}, Reading defined the following poset structure on the set of triangulations of $P(Q)$. A  triangulation $\calt_1$ is said to be {\em covered by} a triangulation $\calt_2$  if there exist two diagonals $\zg_1\ne\zg_2$  such that $\calt_1\setminus\{\zg_1\}=\calt_2\setminus\{\zg_2\}$ and 
the slope of $\zg_1$ is smaller than the slope of $\zg_2$. 
This covering relation induces a partial order on the set of triangulations, and the resulting poset is called the \emph{Cambrian lattice of type} $\mathbb{A}$. 
See also the survey paper~\cite{ReadingBeyond}.

\begin{definition}
 \label{def poset}
 We define a partial order on $\mar$ as follows. 
 For $T_1,T_2\in \mar$, we say that~$T_1 $ is \emph{covered by} $T_2$ if there exist indecomposable summands $M_i$ of $T_i$ such that $T_1/M_1\cong T_2/M_2$ and there is a non-split short exact sequence
 \[\xymatrix{0\ar[r]& M_1\ar[r]^(0.4)f& E\oplus E'\ar[r]^(0.6)g &M_2\ar[r] &0}\]
 with $E,E'$ indecomposable summands of $T_1/M_1$.
\end{definition}

\begin{remark}
 In the short exact sequence in Definition~\ref{def poset}, the morphism $f$ is a minimal $\add (T_1/M_1)$ approximation of $M_1$. 
 For a definition, see for example~\cite{BMRRT}. 
\end{remark}
 
\begin{remark}
 It follows from Theorem~\ref{thm 2} that every maximal almost rigid representation contains all hooks and all cohooks of $Q$, since these correspond to boundary edges of the triangulation of~$P(Q)$ by Corollary~\ref{cor 38}. 
Therefore, there is a unique minimal element in the poset $\mar$ given by the basic representation whose  direct summands are  the indecomposable projective representations together
 with the hooks and cohooks.
 The covering relation given by the short exact sequence in Definition~\ref{def poset} then gives an algorithm to construct all maximal almost rigid representations using approximations. 
 The unique maximal element in the poset on $\mar$ is given by the basic representation whose  direct summands are  the indecomposable injective representations together
 with the hooks and cohooks.
\end{remark}
 
 \subsection{The poset on $\mar$ is a Cambrian lattice}

We have the following result.

\begin{thm}
\label{thm 4}
Let $Q$ be a Dynkin quiver of type $\mathbb{A}_{n+2}$. Then the poset on $\mar$ is isomorphic to a Cambrian lattice.
\end{thm}

\begin{proof}
By Theorem \ref{thm 2}, the functor $F$  is a bijection from triangulations of $P(Q)$ to $\mar$. 
 It only remains to prove that $F$ preserves the covering relations. This is done in the next lemma.
\end{proof}

\begin{lemma}\label{lem 7} Let $Q$ be a quiver of type $\mathbb{A}_{n+2}$. Let $M_1, M_2$ be two indecomposable representations  of $Q$ and $\zg_1,\zg_2$ be the line segments in $P(Q)$ such that $F(\zg_i)=M_i$, $i = 1,2$.
Suppose there is a non-split short exact sequence $0\to M_1\to E\to M_2\to 0$. Then the following conditions are equivalent.
\begin{enumerate}
\item 
$\Hom(M_1,M_2)\ne 0$.
\item $E$ is decomposable.
\item $\zg_1$ and $\zg_2$ cross and
the slope of $\gamma_2$ is larger than the slope of $\gamma_1$. 
\end{enumerate}
\end{lemma}

\begin{proof}
(1)$\ssi$(2). Applying the functor $\Hom(M_1,-)$ to the short exact sequence yields an exact sequence 
\begin{equation}\label{eq 7}0\to\kb=\Hom(M_1,M_1)\to\Hom(M_1,E)\to\Hom(M_1,M_2)\to \Ext(M_1,M_1)=0,\end{equation}
where the first term is one-dimensional and the last term is zero, because $M_1$ is an indecomposable representation of a Dynkin quiver. 
Since $M_1$ and $M_2 $ are indecomposable and the short exact sequence  is non-split, there is a nonzero morphism from $M_1$ to every indecomposable summand of $E$. Thus  the dimension of $\Hom(M_1,E)$ is at least equal to  the number of indecomposable summands of~$E$. However, $Q$ being of type $\mathbb{A}$ implies that the dimension of $\Hom$ between indecomposables is at most one, and thus the dimension of $\Hom(M_1,E)$ is \emph{exactly} equal to  the number of indecomposable summands of~$E$. Thus the exactness of (\ref{eq 7}) implies that  $E$ is decomposable if and only if $\Hom(M_1,M_2)\ne 0$.

(1)\&(2)$\Rightarrow$(3). Suppose there is a nonzero morphism from $M_1$ to $M_2$. 
By Theorem~\ref{thm 1}, this means that we can get from $\gamma_1$ to $\gamma_2$ by a sequence of pivots, each of which is moving one of the endpoints of a diagonal to its counterclockwise neighbor (in a way that preserves the ordering of the endpoints).
Each such pivot increases the slope of a diagonal, thus the slope of $\zg_2$ is larger than the slope of $\zg_1$.
Furthermore $\zg_1$ and $\zg_2$ must cross, because otherwise the extension $E$ would be indecomposable.

(3)$\Rightarrow$(1). Suppose $\zg_1$ and $\zg_2$ cross and that
$\gamma_2$ has larger slope than $\zg_1$. Then 
it is possible to get from $\gamma_1$ to $\gamma_2$
by 
a sequence of counterclockwise pivots which preserve the ordering of the endpoints. By Theorem~\ref{thm 1}, there is a  nonzero morphism from $M_1$ to $M_2$.
\end{proof}

\subsection{The $\eta$ map}
 Recall the map  $\eta_Q$ from the symmetric group $S_{n+1}$ to the set of triangulations of $P(Q)$ from Definition \ref{def etaR}. 
 We define a new realization of this map via the  composition of $\eta_Q$ with the equivalence $F$ of Theorem~\ref{thm 1}.
\begin{definition}
 Let $\eta_Q^{\rep}=F\circ \eta_Q\colon S_{n+1}\to \mar.$\end{definition}
 
We shall give an alternate, equivalent description of $\eta_Q^\rep$ using extensions and  degenerations. %
A representation $M$ is an \emph{extension} of a representation $N$ if there exists a non-split short exact sequence
$0\to N_1\to M\to N_2\to 0$ such that $N=N_1\oplus N_2$. 
A representation $M$ is a \emph{degeneration} of a representation $N$ if $M$ is contained  in the Zariski closure of the isomorphism class of $N$. In Dynkin type $\mathbb{A}$, 
a representation $M$ is a degeneration of a representation $N$ 
if there exists a non-split short exact sequence $0\to M_1\to N\to M_2\to 0$ with $M=M_1\oplus M_2$. 
Note that in both cases, the dimension vectors of $M$ and $N$ are equal. 

Recall that the map $\eta_Q$ is defined by taking the union of a list of paths. 
This list is created by a recursive process of removing and adding a vertex of $P(Q)$ to a path. 
Instead of a list of paths, we will now create a list of representations of $Q$ each of which has dimension vector  $(1,1,\dots,1)$. 
Instead of removing (respectively adding) a vertex, we will now apply an extension (respectively a degeneration) to a representation. Because every representation on this list has dimension vector $(1,1,\dots,1)$, an extension is the same as changing the assignment of the linear map on an arrow in $Q$ from the 
zero map to the identity map. Conversely, a degeneration is simply changing the assignment of the linear map on an arrow in $Q$ from the 
 identity map to the zero map. 
 Finally, instead of taking the union of the list of paths, we take the direct sum of the union of all indecomposable summands of our list of representations.

Let $\alpha_\ell$ denote the arrow of $Q$ between vertices $\ell$ and $(\ell+1)$ for all $\ell$.
Given a permutation $\pi=\pi_1 \, \pi_2 \, \pi_3 \, \dots \, \pi_{n+1}$ written in one-line notation, 
we define a representation $\lambda_i^\rep(\pi)$ of $Q$ for each $i\in \{0,\dots,n+1\}$. 
Let $\lambda_0^\rep(\pi)$ be the 
 representation with dimension $(1,1,\dots,1)$ with 
the $0$ map on each arrow $\alpha_i$ for all    $\underline{i}\in\underline{[n+1]}$, and the identity map everywhere else. 
That is, it is the direct sum of the indecomposable representations which correspond to the lower boundary edges of $P(Q)$ (or maximal increasing paths in $Q$, due to Lemma~\ref{lemma:maximal_paths_to_boundary_edges}). These indecomposable representations are the ones in the $\tau$-orbit of the projective at $1$. In the Auslander--Reiten quiver example of  Figure~\ref{fig:ar_quiver}, they are drawn at the bottom row. 

Define $\lambda^{\rep}_1(\pi)$ to be 
\begin{itemize}
\item[(a)] the degeneration of $\lambda^{\rep}_0(\pi)$ obtained by
replacing the identity map on $\alpha_{\pi_1}$ with the zero map, if $\pi_1\in\overline{[n+1]}$, or
\item[(b)] the extension of $\lambda^{\rep}_0(\pi)$ obtained by
replacing the zero map on $\alpha_{\pi_1}$ with the identity map, if $\pi_1\in\underline{[n+1]}$. 
\end{itemize}

Repeat this process recursively, until we get to $\lambda_{n+1}^{\rep}(\pi)$ which is the  representation with dimension vector $(1,1,\dots,1)$ with 
the $0$ map on all arrows $\alpha_i$ with    $\overline{i}\in\overline{[n+1]}$, and the identity map everywhere else. That is, $\lambda_{n+1}^{\rep}(\pi)$ is the direct sum of the indecomposable representations which correspond to the upper boundary edges of $P(Q)$ (or maximal decreasing paths in $Q$, due to Lemma~\ref{lemma:maximal_paths_to_boundary_edges}). These indecomposable representations are the ones in the $\tau$-orbit of the projective at $n+2$. In the Auslander--Reiten quiver example of  Figure~\ref{fig:ar_quiver}, they are drawn at the top row. 

\begin{prop}
 $\eta^{rep}_Q(\pi)$ is the maximal almost rigid 
representation of $Q$ whose indecomposable direct summands are exactly those appearing in 
the list $\lambda^{\rep}_0(\pi), \dots, \lambda^{\rep}_{n+1}(\pi)$. 
\end{prop}
\begin{proof}
 We need to check that the recursive definition of $\zl_\ell^\rep$ corresponds to the recursive definition of $\zl_\ell$ under the functor $F$. 
 For $\ell=0$, this is clear. Suppose $\ell>0$ and  that $F(\zl_\ell)=\zl_\ell^\rep$. We want to show $F(\zl_{\ell+1})=\zl_{\ell+1}^\rep$.  
 Suppose first that $\pi_{\ell+1}\in\overline{[n+1]}$.
 Then the difference between the paths $\zl_\ell$ and $\lambda_{\ell+1}$ is that $\zl_{\ell+1}$ passes through the vertex $\pi_{\ell+1} $ of $P(Q)$. 
 This means that $\zl_{\ell} $ contains a line segment $\zg(i,j)$ which is replaced in $\zl_{\ell+1}$ by two line segments $\zg(i,\pi_{\ell+1}),\zg(\pi_{\ell+1},j)$, where $i<\pi_{\ell+1}<j$. 
 Applying the functor $F$, the indecomposable direct summand  $M(i+1,j)$  is replaced by two indecomposable direct summands $M(i+1,\pi_{\ell+1}),M(\pi_{\ell+1}+1,j)$.
  This change coincides with the difference between $\zl^{\rep}_{\ell}$ and $\zl^{\rep}_{\ell+1}$ which is given by a degeneration that replaces the identity map on the arrow $\za_{\pi_{\ell+1}}$ by the zero map.
  
  Now suppose that  $\pi_{\ell+1}\in\underline{[n+1]}$.
 Then the difference between the paths $\zl_\ell$ and $\lambda_{\ell+1}$ is that~$\zl_{\ell}$ passes through the vertex $\pi_{\ell+1} $ of $P(Q)$. This means that $\zl_{\ell} $ contains two line segments $\zg(i,\pi_{\ell+1}),\zg(\pi_{\ell+1},j)$, with $i<\pi_{\ell+1}<j$, which are replaced in $\zl_{\ell+1}$ by a single line segment $\zg(i,j)$. Applying the functor $F$, two indecomposable direct summands $M(i+1,\pi_{\ell+1}),M(\pi_{\ell+1}+1,j)$ are replaced by a single indecomposable direct summand  $M(i+1,j)$ .
  This change coincides with the difference between $\zl^{\rep}_{\ell}$ and $\zl^{\rep}_{\ell+1}$ which is given by an extension that replaces the zero map on the arrow $\za_{\pi_{\ell+1}}$ by the identity map.
\end{proof}

\begin{example}
\label{exam:eta_rep}
Let $Q$ be the quiver in Figure~\ref{fig:polygon6}. Then $\overline{[6]}=\{1,2,3,5\}$ and $\underline{[6]}=\{4,6\}$. 
Let $\pi=
4 \, 5 \,3 \, 1 \, 2 \, 6
\in S_{6}$.  The paths $\zl_i(\pi)$ are listed in Example~\ref{exam:eta}.
The list of the representations $\zl_i^\rep(\pi)$ is as follows; see also Figure \ref{fig:booleanalgebra}.
\[\begin{array}
 {rcl} 
 \zl_0^\rep&=&M(1,4)\oplus M(5,6)\oplus S(7)\\
 \zl_1^\rep&=&M(1,6)\oplus S(7)\\
 \zl_2^\rep&=&M(1,5)\oplus S(6)\oplus S(7)\\\zl_3^\rep&=&M(1,3)\oplus M(4,5)\oplus S(6)\oplus S(7)\\
 \zl_4^\rep&=&S(1)\oplus M(2,3)\oplus M(4,5)\oplus S(6)\oplus S(7)\\
  \zl_5^\rep&=&S(1)\oplus S(2)\oplus S(3)\oplus M(4,5)\oplus S(6)\oplus S(7)\\
  \zl_6^\rep&=&S(1)\oplus S(2)\oplus S(3)\oplus M(4,5)\oplus M(6,7)
\end{array}
\]
As we have seen earlier the fiber of $\eta_Q(453126)$ is the set 
    $ \left\{ 
453126, 
453162, 
453612, 
456312  
    \right\}.$
    These four permutations correspond to the four maximal chains in Figure~\ref{fig:booleanalgebra}.
\end{example}

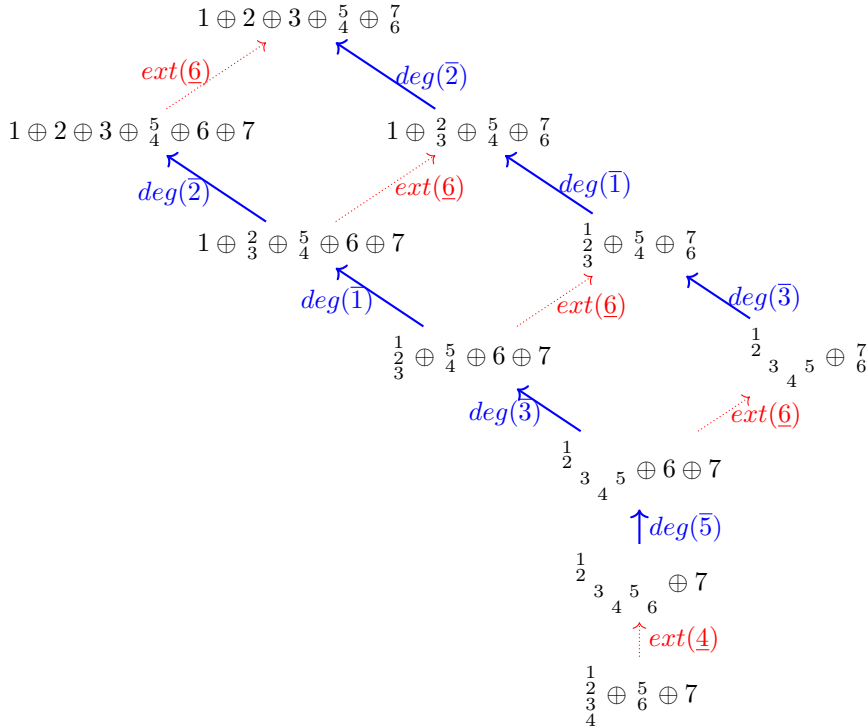
\begin{figure}[h!t]
\def\xfigscale{2.25}
\def\yfigscale{1.5}
\begin{tikzpicture}[xscale=\xfigscale,yscale=\yfigscale]
\def\extensionedge{densely dotted}
\def\degenerationedge{thick}
\def\extensioncolor{red}
\def\degenerationcolor{blue}

\node[](46) at (0,0) {$\begin{smallmatrix}1\\2\\3\\4\end{smallmatrix}\oplus \begin{smallmatrix}5\\6\end{smallmatrix} \oplus 7$};

\node[](6) at (0,1) {$\begin{smallmatrix}1\\2\\&3&&5 \\&&4&&6 \end{smallmatrix} \oplus 7$};

\node (56) at (0,2) {$\begin{smallmatrix}1\\2\\&3&&5 \\&&4 \end{smallmatrix} \oplus 6 \oplus 7$};

\node (356) at (-1,3) {$\begin{smallmatrix}1\\2\\3\end{smallmatrix}
\oplus 
\begin{smallmatrix}5\\4 \end{smallmatrix} 
\oplus 6 \oplus 7$};
\node (5) at (1,3) {$\begin{smallmatrix}1\\2\\&3&&5 \\&&4\end{smallmatrix} 
\oplus 
\begin{smallmatrix}7\\6 \end{smallmatrix}$};

\node (1356) at (-2,4) {$1\oplus \begin{smallmatrix}2\\3 \end{smallmatrix} \oplus \begin{smallmatrix}5\\4 \end{smallmatrix} \oplus 6 \oplus 7$};
\node (35) at (0,4) {$
\begin{smallmatrix}1\\2\\3 \end{smallmatrix} 
\oplus 
\begin{smallmatrix}5\\4 \end{smallmatrix} 
\oplus 
\begin{smallmatrix}7\\6 \end{smallmatrix}$};

\node (12356) at (-3,5) {$1\oplus 2 \oplus 3 \oplus \begin{smallmatrix}5\\4 \end{smallmatrix} \oplus 6 \oplus 7$};
\node (135) at (-1,5) {$1\oplus 
\begin{smallmatrix}2\\3 \end{smallmatrix} 
\oplus 
\begin{smallmatrix}5\\4 \end{smallmatrix} 
\oplus 
\begin{smallmatrix}7\\6 \end{smallmatrix}$};

\node (1235) at (-2,6) {$1\oplus 2 \oplus 3 \oplus \begin{smallmatrix}5\\4 \end{smallmatrix} \oplus \begin{smallmatrix}7\\6 \end{smallmatrix}$};

\draw[\extensioncolor, \extensionedge, ->](46)
 -- (6) node[\extensioncolor,pos=0.5,right,right] {$ext(\underline{4})$};
 
 \draw[\degenerationcolor, \degenerationedge,->](6)
 -- (56) node[\degenerationcolor,pos=0.5,right,right] {$deg(\overline{5})$};
 
  \draw[\degenerationcolor,\degenerationedge, ->](56)
 -- (356) node[\degenerationcolor,pos=0.45,right,left] {$deg(\overline{3})$};
  \draw[\extensioncolor,\extensionedge,  ->](56)
 -- (5) node[\extensioncolor,  pos=0.45,right,right] {$ext(\underline{6})$};
 
  \draw[\degenerationcolor, \degenerationedge, ->](356)
 -- (1356) node[\degenerationcolor,pos=0.45,right,left] {$deg(\overline{1})$};
   \draw[\extensioncolor, \extensionedge, ->](356)
 -- (35) node[\extensioncolor,pos=0.40,right,right] {$ext(\underline{6})$};
\draw[\degenerationcolor, \degenerationedge, ->](5)
 -- (35) node[\degenerationcolor,pos=0.5,right,right] {$deg(\overline{3})$}; 
 
  \draw[\degenerationcolor, \degenerationedge, ->](1356)
 -- (12356) node[\degenerationcolor,pos=0.45,right,left] {$deg(\overline{2})$}; 
   \draw[\extensioncolor, \extensionedge, ->](1356)
 -- (135) node[\extensioncolor,pos=0.5,right,right] {$ext(\underline{6})$}; 
 \draw[\degenerationcolor, \degenerationedge, ->](35)
 -- (135) node[\degenerationcolor,pos=0.5,right,right] {$deg(\overline{1})$}; 
 
   \draw[\extensioncolor, \extensionedge, ->](12356)
 -- (1235) node[\extensioncolor,pos=0.55,right,left] {$ext(\underline{6})$}; 
    \draw[\degenerationcolor, \degenerationedge, ->](135)
 -- (1235) node[\degenerationcolor,pos=0.5,right,right] {$deg(\overline{2})$}; 
\end{tikzpicture}
\caption{The four maximal chains correspond to the four permutations in the fiber of $\eta_Q^\rep(\pi)$ given in Example~\ref{exam:eta_rep}. The left-most maximal chain corresponds to the permutation $453126$.}
\label{fig:booleanalgebra}
\end{figure}

\subsection*{Acknowledgements} 
We thank Bernhard Keller for suggesting an outline for proving Theorem~\ref{thm A}. 
We also thank Nathan Reading and Hugh Thomas
for answering our questions about Cambrian lattices,
and Ana Garcia Elsener, Kaveh Mousavand, Pierre-Guy Plamondon, and Gordana Todorov for helpful suggestions. 
Special thanks  go to Alastair King for stimulating discussions on stability. 
Finally, we thank the anonymous reviewer whose suggestions helped improve and clarify this paper.


\begin{thebibliography}{BMRRT06}

\bibitem[ABCJP10]{ABCP}
I.~Assem, T.~Br\"{u}stle, G.~{C}harbonneau Jodoin, and P-G. Plamondon.
\newblock Gentle algebras arising from surface triangulations.
\newblock {\em Algebra Number Theory}, 4(2):201--229, 2010.

\bibitem[ABS08a]{ABS2}
I.~Assem, T.~Br\"{u}stle, and R.~Schiffler.
\newblock Cluster-tilted algebras and slices.
\newblock {\em J. Algebra}, 319(8):3464--3479, 2008.

\bibitem[ABS08b]{ABS}
I.~Assem, T.~Br\"{u}stle, and R.~Schiffler.
\newblock Cluster-tilted algebras as trivial extensions.
\newblock {\em Bull. Lond. Math. Soc.}, 40(1):151--162, 2008.

\bibitem[AI19]{AI}
P.~Apruzzese and K.~Igusa.
\newblock Stability conditions for affine type {A}.
\newblock {\em Algebras and Representation Theory}, pages 1--33, 2019.

\bibitem[ASS06]{ASS}
I.~Assem, D.~Simson, and A.~Skowro\'{n}ski.
\newblock {\em Elements of the representation theory of associative algebras.
  {V}ol. 1}, volume~65 of {\em London Mathematical Society Student Texts}.
\newblock Cambridge University Press, Cambridge, 2006.
\newblock Techniques of representation theory.

\bibitem[BC21]{BCS21}
K.~Baur and R.~{Coelho Sim\~{o}es}.
\newblock A geometric model for the module category of a gentle algebra.
\newblock {\em Int. Math. Res. Not. IMRN}, 2021(15):11357--11392, 2021.

\bibitem[BMRRT06]{BMRRT}
A.~Buan, R.~Marsh, M.~Reineke, I.~Reiten, and G.~Todorov.
\newblock Tilting theory and cluster combinatorics.
\newblock {\em Adv. Math.}, 204(2):572--618, 2006.

\bibitem[BMR07]{BMR}
A.~Buan, R.~Marsh, and I.~Reiten.
\newblock Cluster-tilted algebras.
\newblock {\em Trans. Amer. Math. Soc.}, 359(1):323--332, 2007.

\bibitem[BR87]{BR}
M.~C.~R. Butler and C.~M. Ringel.
\newblock Auslander-{R}eiten sequences with few middle terms and applications
  to string algebras.
\newblock {\em Comm. Algebra}, 15(1-2):145--179, 1987.

\bibitem[Bri07]{Bridgeland}
T.~Bridgeland.
\newblock Stability conditions on triangulated categories.
\newblock {\em Ann. of Math. (2)}, 166(2):317--345, 2007.

\bibitem[BS94]{BS}
L.~Billera and B.~Sturmfels.
\newblock Iterated fiber polytopes.
\newblock {\em Mathematika}, 41(2):348--363, 1994.

\bibitem[BW97]{BW}
A.~Bj\"{o}rner and M.~Wachs.
\newblock Shellable nonpure complexes and posets. {II}.
\newblock {\em Trans. Amer. Math. Soc.}, 349(10):3945--3975, 1997.

\bibitem[CCS06]{CCS}
P.~Caldero, F.~Chapoton, and R.~Schiffler.
\newblock Quivers with relations arising from clusters ({$A_n$} case).
\newblock {\em Trans. Amer. Math. Soc.}, 358(3):1347--1364, 2006.

\bibitem[Gab72]{Gabriel}
P.~Gabriel.
\newblock Unzerlegbare {D}arstellungen. {I}.
\newblock {\em Manuscripta Math.}, 6:71--103; correction, ibid. 6 (1972), 309,
  1972.

\bibitem[Hap88]{H}
D.~Happel.
\newblock {\em Triangulated categories in the representation theory of
  finite-dimensional algebras}, volume 119 of {\em London Mathematical Society
  Lecture Note Series}.
\newblock Cambridge University Press, Cambridge, 1988.

\bibitem[HR82]{HR}
D.~Happel and C.~M. Ringel.
\newblock Tilted algebras.
\newblock {\em Trans. Amer. Math. Soc.}, 274(2):399--443, 1982.

\bibitem[IT09]{IT}
C.~Ingalls and H.~Thomas.
\newblock Noncrossing partitions and representations of quivers.
\newblock {\em Compos. Math.}, 145(6):1533--1562, 2009.

\bibitem[Kin94]{King}
A.~D. King.
\newblock Moduli of representations of finite-dimensional algebras.
\newblock {\em Quart. J. Math. Oxford Ser. (2)}, 45(180):515--530, 1994.

\bibitem[Kin22]{Kinser}
R.~Kinser.
\newblock Total {S}tability {F}unctions for {T}ype {$\mathbb{A}$} {Q}uivers.
\newblock {\em Algebr. Represent. Theory}, 25(4):835--845, 2022.

\bibitem[LF09]{L}
D.~Labardini-Fragoso.
\newblock Quivers with potentials associated to triangulated surfaces.
\newblock {\em Proc. Lond. Math. Soc. (3)}, 98(3):797--839, 2009.

\bibitem[OPS18]{OPS}
S.~Opper, P-G. Plamondon, and S.~Schroll.
\newblock A geometric model for the derived category of gentle algebras, 2018.
\newblock Preprint \arxiv{1801.09659}.

\bibitem[Qiu15]{Qiu15}
Y.~Qiu.
\newblock Stability conditions and quantum dilogarithm identities for {D}ynkin
  quivers.
\newblock {\em Adv. Math.}, 269:220--264, 2015.

\bibitem[Qiu18]{Qiu}
Y.~Qiu.
\newblock Global dimension function on stability conditions and {G}epner
  equations, 2018. Preprint \arxiv{1807.00010}.

\bibitem[QZ22]{QZ22}
Y.~Qiu and X.~Zhang.
\newblock Geometric classification of total stability spaces, 2022.
\newblock Preprint \arxiv{2202.00092}.

\bibitem[Rea06]{Reading}
N.~Reading.
\newblock Cambrian lattices.
\newblock {\em Adv. Math.}, 205(2):313--353, 2006.

\bibitem[Rea12]{ReadingBeyond}
N.~Reading.
\newblock From the {T}amari lattice to {C}ambrian lattices and beyond.
\newblock In {\em Associahedra, {T}amari lattices and related structures},
  volume 299 of {\em Prog. Math. Phys.}, pages 293--322.
  Birkh\"{a}user/Springer, Basel, 2012.

\bibitem[Rei02]{Reiner}
V.~Reiner.
\newblock Equivariant fiber polytopes.
\newblock {\em Doc. Math.}, 7:113--132, 2002.

\bibitem[Rei03]{Reineke}
M.~Reineke.
\newblock The {H}arder-{N}arasimhan system in quantum groups and cohomology of
  quiver moduli.
\newblock {\em Invent. Math.}, 152(2):349--368, 2003.

\bibitem[Rie80]{R}
C.~Riedtmann.
\newblock Algebren, {D}arstellungsk\"{o}cher, \"{U}berlagerungen und
  zur\"{u}ck.
\newblock {\em Comment. Math. Helv.}, 55(2):199--224, 1980.

\bibitem[RS09]{RS}
N.~Reading and D.~Speyer.
\newblock Cambrian fans.
\newblock {\em J. Eur. Math. Soc. (JEMS)}, 11(2):407--447, 2009.

\bibitem[Rud97]{Rudakov}
A.~Rudakov.
\newblock Stability for an abelian category.
\newblock {\em J. Algebra}, 197(1):231--245, 1997.

\bibitem[Sch91]{Schofield}
A.~Schofield.
\newblock Semi-invariants of quivers.
\newblock {\em J. London Math. Soc. (2)}, 43(3):385--395, 1991.

\bibitem[Sch08]{S2}
R.~Schiffler.
\newblock A geometric model for cluster categories of type {$D_n$}.
\newblock {\em J. Algebraic Combin.}, 27(1):1--21, 2008.

\bibitem[Sch14]{S}
R.~Schiffler.
\newblock {\em Quiver representations}.
\newblock CMS Books in Mathematics/Ouvrages de Math\'{e}matiques de la SMC.
  Springer, Cham, 2014.

\bibitem[Sta12]{Stanley}
R.~Stanley.
\newblock {\em Enumerative combinatorics. {V}olume 1}, volume~49 of {\em
  Cambridge Studies in Advanced Mathematics}.
\newblock Cambridge University Press, Cambridge, second edition, 2012.

\bibitem[Ton97]{Tonks}
A.~Tonks.
\newblock Relating the associahedron and the permutohedron.
\newblock In {\em Operads: {P}roceedings of {R}enaissance {C}onferences
  ({H}artford, {CT}/{L}uminy, 1995)}, volume 202 of {\em Contemp. Math.}, pages
  33--36. Amer. Math. Soc., Providence, RI, 1997.

\end{thebibliography}
\end{document}